\documentclass{lms}

\usepackage{graphicx} 
\usepackage{psfrag}
\usepackage{amsmath}

\newtheorem{Theorem}{Theorem}[section]
\newtheorem{Lemma}[Theorem]{Lemma}
\newtheorem{Corollary}[Theorem]{Corollary}
\newtheorem{Proposition}[Theorem]{Proposition}

\usepackage{amssymb}
\usepackage{mathrsfs}

\newnumbered{Definition}[Theorem]{Definition}
\newnumbered{Remark}[Theorem]{Remark}
\newunnumbered{notation}{Notation}
\newunnumbered{proof of claim}{Proof of claim}
\newunnumbered{sa}{Standing assumption}

\def\F{{\mathcal{F}}}
\def\J{{\mathcal{J}}}
\def\L{{\mathscr{L}}}
\def\B{{\mathscr{B}}}
\def\C{{\mathbb{C}}}
\def\D{{\mathbb{D}}}
\def\N{{\mathbb{N}}}

\providecommand{\dist}{\mathop{\rm dist}\nolimits}

\providecommand{\e}{\mathop{\rm e}\nolimits}
\providecommand{\Attr}{\mathop{\rm Attr}\nolimits}
\providecommand{\Par}{\mathop{\rm Par}\nolimits}

\newcommand{\eps}{\varepsilon}
\newcommand{\Oo}{\operatorname{O}}

\title[Landing of geometrically finite entire maps]{A landing theorem for dynamic rays of geometrically finite entire functions} 

\author{Helena Mihaljevi\'{c}-Brandt}

\classno{37F10 (primary), 30F45, 37B10, 30D05, 37F20 (secondary).}

\extraline{Supported by the Engineering and Physical Sciences Research Council (EPSRC), grant-code: EP/E052851.}

\begin{document}
\maketitle

\begin{abstract}
 A transcendental entire function $f$ is called \emph{geometrically finite}
  if the intersection of the set $S(f)$
  of singular values with the Fatou set $\F(f)$ is compact
  and the intersection of the postsingular set $P(f)$ with
  the Julia set $\J(f)$ is finite. (In particular,
  this includes all entire functions with finite postsingular set.) 
  If $f$ is geometrically finite, then $\F(f)$ is either empty or consists
  of the basins of attraction of finitely many attracting or parabolic cycles.

 Let $z_0$ be a repelling or parabolic periodic point of such a map $f$. We
  show that, if $f$ has finite order, then
  there exists an injective curve consisting of escaping points of
  $f$ that connects $z_0$ to $\infty$. (This curve is called
  a \emph{dynamic ray}.) In fact, the assumption of finite order can be
  weakened considerably; for example, it is sufficient to assume that
  $f$ can be written as a finite composition of finite-order functions. 
\end{abstract}

\section{Introduction}

In polynomial dynamics, \emph{dynamic rays}, which foliate 
the set of escaping points (points which tend to $\infty$ under iteration), 
were introduced by Douady and Hubbard as a tool in their famous work on the Mandelbrot set \cite{douady}. Since then, dynamic rays, and their landing properties in particular, 
have been an essential ingredient in the success of polynomial dynamics. 
One of the  fundamental results in this area, which goes back to Douady, 
states that each repelling or parabolic (pre)periodic point of a polynomial 
with connected Julia set is the landing point of at least one (pre)periodic 
dynamic ray \cite[Theorem 18.11]{milnor}. 

In the polynomial case, dynamic rays arise naturally as preimages of 
 straight rays under the B\"ottcher isomorphism at $\infty$. When $f$
 is a transcendental entire function, $\infty$ is no longer a superattracting
 fixed point, but rather an essential singularity. Hence the \emph{escaping set}
   \[ I(f) := \{z\in\C: f^n(z)\to\infty\} \]
 is no longer open. 
 Nevertheless, it has long been known \cite{devaney,devaney2} 
 that for certain classes of entire functions there exist curves in the 
escaping set which can be seen as analogs of dynamic rays. We shall refer to these
 curves (which are also often known as \emph{hairs}) themselves as
 dynamic rays of $f$ to stress the analogy to the polynomial case. (See
 Definition \ref{defn_ray} for a formal definition of 
 dynamic rays, and also below for the special case of \emph{periodic} rays
 that is of main interest to us.)

Recently Rottenfu\ss er, R\"{u}ckert, Rempe and Schleicher \cite{rrrs} 
 proved that, for an entire function that is a composition of 
 finite-order entire functions with 
 bounded sets of singular values, the escaping set
 consists of dynamic rays and (some of) their endpoints. 
 This provides us with a large class of functions where we can study 
 the topology of Julia sets by looking at landing properties of dynamic rays. 
 This approach to the study of Julia sets and escaping sets has been used with 
 great success in certain families of entire transcendental maps like the 
 exponential family $E_{\lambda}(z)=\lambda e^{z}$ or the cosine family 
 $F_{a,b}(z)=a e^z + b e^{-z}$. 

To state the main question of interest to us, let us begin by 
 formally defining what we mean by a \emph{periodic (dynamic) ray}.

\begin{Definition}[(Periodic rays)] \label{defn_periodic_rays}
  Let $f$ be an entire function. Then a \emph{periodic ray} of $f$
   is an injective curve
    \[ g:(0,\infty)\to\C \]
   such that $\lim_{t\to \infty}g(t)=\infty$ and such that
   there is an $n\geq 1$ with
   $g(2t)=f^n(g(t))$ for all $t$. The minimal such $n$ is called
   the \emph{period} of $g$. 

  As usual, we say that $g$ \emph{lands} at $z_0$
   if $\lim_{t\to 0} g(t)=z_0$. 
\end{Definition}
 
A natural question suggested by Douady's theorem and the results
 of \cite{rrrs} is the following. Suppose that $f$ is a finite-order
 entire function whose singular set is bounded, and suppose also that 
 $S(f)\cap I(f)=\emptyset$. Is 
 every repelling or parabolic periodic point of $f$ the landing
 point of a periodic ray of $f$?

Even for exponential and cosine maps, this question is still open.
 (For a partial result on exponential maps, compare  \cite{rempe2}.)
 However, Schleicher and Zimmer \cite{schleicher1}
 obtained a positive answer for exponential maps satisfying certain
 dynamical assumptions. 
 In this paper, we generalize this statement, under similar conditions,
 to a much larger class of entire functions. 

 For an entire transcendental function $f$, $S(f)$ denotes the set of
  singular values, $P(f)$ the postsingular set, and $\F(f)$ and $\J(f)$
  the Fatou and Julia sets of $f$, respectively. (For definitions, see
  Section \ref{notations_section}.) 

\begin{Definition}
\label{def_gf}
A map $f$ is called \emph{geometrically finite} if  $S(f)\cap\F(f)$ is compact
 and $P(f)\cap\J(f)$ is finite.
\end{Definition}

We can now state our main result. Recall that $f$ has \emph{finite order} if 
 $\log\log|f(z)|=O(\log|z|)$ as $z\to\infty$.
\begin{Theorem}
\label{thm1}
Let $f$ be geometrically finite and assume that $f$ has finite order.
Then, for any repelling or parabolic periodic point $z$ of $f$, 
 there is a periodic ray landing at $z$.
\end{Theorem}
\begin{Remark*}
In fact, the condition that $f$ is geometrically finite can be weakened;
what we need is that every iterate of $f$ 
has an ``admissible expansion domain'' at every 
repelling or parabolic fixed point (see Definition \ref{dfn_adm_exp_dom}).

 Furthermore, the assumption of finite order is not essential; 
we require that ``all periodic rays exist''. 
(See Theorem \ref{maintheorem} for the precise statement.) 
By \cite{rrrs}, this condition is satisfied
by all finite compositions of finite-order functions with bounded sets
of singular values. 
\end{Remark*}

Our result implies in particular that each singular value in the Julia set of 
 a map $f$ to which the theorem applies 
 is the landing point of some periodic dynamic ray. 
 Using these rays, one can 
 define a dynamically natural partition of the Julia set, as done
 for exponential and cosine maps in \cite{schleicher1,schleicher2}, which
 is useful for studying the 
 topological dynamics of $f$ in combinatorial terms.

We remark that, under more restrictive function-theoretic conditions,
 our result can be considerably strengthened. Indeed, if
 $f=F_{a,b}$ is a cosine map for which both critical values are
 strictly preperiodic, then Schleicher showed in
 \cite{schleicher2} that \emph{every} point $z\in\C$ is either on a 
 dynamic ray or the landing point of a dynamic ray. In \cite{mb} we 
 generalize this result to any finite-order \emph{subhyperbolic} function $f$
 such that $\J(f)$ contains no asymptotic values and the
 degree of critical points in $\J(f)$ is uniformly bounded.
 (A subhyperbolic entire function is a geometrically finite map without
  parabolic cycles.) Furthermore, the Julia set of such a function
  can be described as a ``pinched Cantor bouquet'' in the sense of
  \cite[p.24]{rempe3} (see also \cite[Definition 1.2]{ao}). 

 Here the assumption that there are no asymptotic values in the Julia set
  is essential: indeed, there is \emph{no} exponential map
  $f=E_{\lambda}$ for which
  the asymptotic value belongs to the Julia set and every point of
  $\J(f)$ is either on a dynamic ray or the landing point 
of a dynamic ray \cite{rempe4}.

\subsection*{Idea of the proof}
 In the case of geometrically finite exponential maps, our theorem is due to
  Schleicher and Zimmer \cite{schleicher1} (although in the case where 
  the singular value is preperiodic some of the details are only sketched).
 This was extended to cosine maps with preperiodic critical values in 
  \cite{schleicher2}.  
  The general strategy of our proof, which we will now describe,
  follows the same idea as these papers. 

By passing to a suitable iterate, we can assume that the considered repelling 
 or parabolic periodic point $z_0$ is a fixed point. The idea is to start
  with \emph{any} given curve connecting $z_0$ to infinity, and pull back this 
  curve using iterates of the map $f$. Using hyperbolic contraction arguments, 
  we prove 
  that this procedure yields only finitely many different curves up to 
  homotopy. This then allows us to associate a combinatorial object
  (a ``periodic external address'') to these curves. When we know---e.g. from the existence theorems of \cite{rrrs}---that there exists a periodic ray
  corresponding to this address, it
  then easily follows that this ray lands at $z_0$.

Since the ``ad-hoc'' method that was used 
 to obtain hyperbolic contraction estimates in
 \cite{schleicher1,schleicher2} appears
 to be difficult to adapt to our more general setting, 
 we develop a rather natural construction using hyperbolic 
 geometry. 
 Our argument gives a more streamlined proof even in the established cases. 

We would like to emphasize that the case
when $P(f)$ is finite requires less technical constructions than the general setting. 
This is why we consider this special case separately in the proofs of some results required
for the proof of Theorem \ref{thm1} (see Section \ref{sec_3}). 
 
\begin{acknowledgements}\label{ackref}
I would like to thank especially my supervisor, Lasse Rempe, for his great help and support. 
I would also like to thank Adam Epstein, Freddie Exall and Mary Rees 
for helpful and interesting discussions.
\end{acknowledgements}

\section{Preliminaries}

The complex plane, Riemann sphere and the punctured plane are denoted by 
 $\C$, $\widehat{\C}:=\C\cup\lbrace\infty\rbrace$ and $\C^{*}:=\C\backslash\lbrace 0\rbrace$, respectively. 
 We write $\D$ for the unit disk and  $S^{1}:=\partial\D$ for the unit circle;
 $\D^{*}:=\D\setminus\{0\}$ is the punctured disk. 
 The Euclidean distance between two sets $A, B\subset\C$ 
 is denoted by $\dist(A,B)$. The closure $\overline{A}$ and the
 boundary $\partial A$ of a set $A\subset\C$
 is always understood to be taken relative to the complex plane.

\subsection{Background on holomorphic dynamics} 
\label{notations_section}

Let $f:\C\rightarrow\C$ be an entire transcendental function. Then 
 the set $S(f)$ of \emph{singular values} of $f$ equals the closure 
 of the set of all finite critical and asymptotic values of $f$.
  If $G$ is an open set with $G\cap S(f)=\emptyset$, then the map 
  $f: f^{-1}(G)\rightarrow G$ is a 
  covering map. Denote by 
  $f^n$ the $n$-th iterate of $f$. Then the \emph{postsingular set} of $f$ is 
  defined by  $P(f)=\overline{\cup_{n\geq 0} f^{n}(S(f))}$.

\begin{Lemma}
\label{lem_U_hyp}
Let $f$ be an entire transcendental function. Then $\vert P(f)\vert\geq 2$.
\end{Lemma}

\begin{proof}
Assume first that $f$ has only one singular value, at $w$ say, 
since otherwise the claim already follows. For simplicity, assume $w=0$. 
A  well-known result that follows from covering space theory says that $f$ 
is of the form $f(z)=\exp(az+b)$ for some $a\in\C^{*}, b\in\C$. 
In this case the asymptotic value $w=0$ is also an omitted value, 
so $P(f)$ contains at least two points.
\end{proof}

The \emph{Fatou set} $\F(f)$ of $f$ is the set of all points $z\in\C$ 
so that $(f^n)_{n\in\N}$ forms a normal family in the sense of Montel 
in a neighbourhood of $z$. Its complement $\J(f):=\C\backslash\F(f)$ is called the 
 \emph{Julia set} of $f$. Recall from the introduction that the
 \emph{escaping set} $I(f)$ is the set of 
 all those points $z\in\C$ such that $f^n(z)\rightarrow\infty$ when 
 $n\rightarrow\infty$. We note that $I(f^n)=I(f)$ for all $n\geq 1$. 

To avoid confusion, we call a point $z\in\C$ \emph{preperiodic} 
if some image $f^n(z)$, $n\geq 1$, of $z$ is periodic but not $z$ itself. 
We call a point \emph{eventually periodic} if it is either 
preperiodic or periodic. Let $z$ be a periodic point of $f$ of 
(minimal) period $n$. We call $\mu(z):=(f^n)^{'}(z)$ the 
\emph{multiplier} of $z$. A periodic point $z$ is called 
\emph{attracting} if $0\leq \vert \mu(z)\vert<1$, 
\emph{indifferent} if $\vert \mu(z)\vert=1$ and 
\emph{repelling} if $\vert \mu(z)\vert>1$. 
An attracting periodic point $z$ is called 
\emph{superattracting} if $\mu(z)=0$. We will 
denote the union of all attracting periodic points of $f$ by $\Attr(f)$. 
Since the multiplier of an indifferent periodic point 
is of the from $\e^{2\pi i t}$ with $0\leq t<1$, 
we can distinguish between \emph{rationally} and 
\emph{irrationally indifferent} periodic points, according to 
whether $t$ is rational or not. A rationally indifferent 
periodic point is also called \emph{parabolic}. 
We denote the union of all parabolic cycles of $f$ by $\Par(f)$. 
An irrationally indifferent periodic point in the Julia set is 
called a \emph{Cremer point}. 
The set of all points whose orbits converge to an attracting periodic cycle
 is called the \emph{attracting basin} of this cycle. Likewise, the set of
 points whose orbits converge nontrivially to a parabolic cycle is called the
 \emph{parabolic basin} of that cycle.

Every component of $\F(f)$ is either an \emph{eventually periodic domain} or
a \emph{wandering domain}. The only possible periodic Fatou domains are 
\emph{immediate attracting basins}, \emph{immediate parabolic basins}, 
\emph{Siegel disks} and \emph{Baker domains} (for a detailed explanation 
see \cite[Theorem $6$]{bergweiler1}). If $f$ belongs to the 
\emph{Eremenko-Lyubich class}
\[ \B:=\lbrace f \textrm{ entire transcendental}: 
S(f)\textrm{ is bounded}\rbrace, \]
 then $f$ has no Baker domains, each component of $\F(f)$ is simply-connected 
 and $I(f)\subset J(f)$ (\cite[Proposition 3, Theorem 1]{eremenko1}).

For further background on holomorphic dynamics we refer the reader to \cite{milnor} and \cite{bergweiler1}.

\subsection{Background on hyperbolic geometry}
 A domain (i.e., open connected set) $U\subset\C$ is called \emph{hyperbolic}
  if $\C\setminus U$ contains at least two points. We denote the density
  of the hyperbolic metric on $U$ (i.e., the unique complete conformal
  metric of constant curvature $-1$) by $\rho_{U}(z)$. 
  To each curve $\gamma:(a,b)\rightarrow U$ we assign the 
   \emph{hyperbolic length} 
    $\ell_{U}(\gamma):=\int_{\gamma} \rho_{U}(z)\vert dz\vert$ of 
    $\gamma$. 
 For any two points $z,w\in U$ the 
 \emph{hyperbolic distance} $d_{U}(z,w)$ 
 is the smallest hyperbolic length of a curve connecting 
  $z$ and $w$ in $U$.

 Pick's theorem \cite[Theorem 2.11]{milnor} states that, any holomorphic map $f:V\to U$ 
  between two hyperbolic
  domains does not increase the respective hyperbolic metrics. In fact,
  it is a local isometry if and only if $f$ is a covering map; otherwise 
  $f$ is a strict contraction. 

 In particular, if $V\subsetneq U$, then $\rho_V(z) > \rho_U(z)$ for all
  $z\in V$. 

We will use the following standard estimates on the hyperbolic metric of a
 hyperbolic
 domain $U\subset\C$ 
 \cite[Corollary $A.8$]{milnor}:
\begin{eqnarray}
\label{hyp_est}
\frac{1}{2\cdot\dist(z,\partial U)}\leq\rho_U (z)\leq\frac{2}{\dist(z,\partial U)},
\end{eqnarray}
where the inequality on the left-hand side only holds if $U$ is simply-connected.

By the Uniformization Theorem \cite[Theorem 2.1]{milnor}, any hyperbolic
 domain $U$ is conformally isomorphic to a quotient of the form 
 $\D/\Gamma$, where $\Gamma$ is a Fuchsian group acting on $\D$. 
 Let $\pi:\D\rightarrow U$ be a universal covering map. 

 A \emph{geodesic} on $U$ is the image of a geodesic in $\D$, i.e.\
  an arc of a circle orthogonal to $S^1$. If $g$ is a geodesic connecting 
  two points $z,w\in U$, then $g$ has minimal hyperbolic length among
  curves connecting $z$ and $w$ in the same homotopy class (understood, as usual,
  relative $\partial U$). 
 It will often be important to know that we can replace any curve by 
  a geodesic in the same homotopy class.

\begin{Proposition}[{\cite[Lemma 3]{royden}}]
\label{roy}
 Suppose that $\gamma:[0,\infty]\to\overline{U}$ is a curve with 
  $\gamma\bigl((0,\infty)\bigr)\subset U$. Then there exists a unique geodesic
  $g$ of $U$ that is homotopic to $\gamma$. 
\end{Proposition}

Suppose that $w$ is an isolated point of $\partial U$; such a point is  called
 a \emph{puncture}. By \cite[Proposition 3.8.9]{hubbard}, there exists a
 covering map $p:\D^{*}\rightarrow U$ such that
 $p$ extends to a continuous map $\D\to \C$, sending $0$ to $w$, and such
 that, for sufficiently small
 $\eps>0$, the restriction
 $p:\overline{\D_{\eps}(0)}\setminus\{0\}\to U$ is one-to-one.
 If $\eps$ has this property, then the simple closed curve 
  $h_{\eps}(w) := p(\partial\D_{\eps}(0))$ is called a 
  \emph{horocycle at $w$}; the component $H_{\eps}(w)$ of $U\setminus h_{\eps}(w)$
  whose boundary is $h_{\eps}(w)\cup\{w\}$ is called a \emph{horosphere at $w$}.
 
The following result states that geodesics on hyperbolic 
domains will stay away from the punctures.

\begin{Lemma}[{\cite[Lemma 2]{jorgensen}}]
\label{jorg}
Let $U$ be a hyperbolic domain and let $w$ be a puncture of $U$. 
 There exists $\eps>0$ such that each simple geodesic entering
 the horosphere $H_{\eps}(w)$ ends at the point $w$. 
\end{Lemma}

\subsection{Tracts and external addresses}
\label{subs_tracts}
Next, we want to review the constructions of tracts and external addresses for 
a function $f$ with bounded postsingular set, following \cite{rott,rrrs}. 

Let $D$ be any Jordan domain containing $P(f)$ and define 
$A:=\C\backslash\overline{D}$ and $G:= f^{-1}(A)$. Then each 
component $T$ of $G$ is a simply-connected unbounded domain 
whose boundary $\partial T$ is a Jordan arc tending to 
 $\infty$ at both ends. Such a component  $T$ is called a 
\emph{tract} of the function $f$; the restriction 
$f: T\rightarrow A$ is a universal covering
\cite[Theorem 1.1]{devaney}. There can be only finitely many 
tracts having non-empty intersection with $D$. 

Next, we choose a curve $\alpha\subset A$ not intersecting any 
tract $T$, such that $\alpha$ connects $\partial D$ to $\infty$. 
The preimage $f^{-1}(\alpha)$ induces a partition of each tract, 
cutting it into countably many components called \emph{fundamental domains}. 
Every such domain is mapped conformally to $A\backslash{\alpha}$ under $f$.

\begin{figure}
\centering
\psfrag{F}{$F$}
\psfrag{f}{$f$}
\psfrag{A}{$A$}
\psfrag{D}{$D$}
\psfrag{T}{$T$}
\psfrag{alfa}{$\alpha$}
 \includegraphics[width=\textwidth]{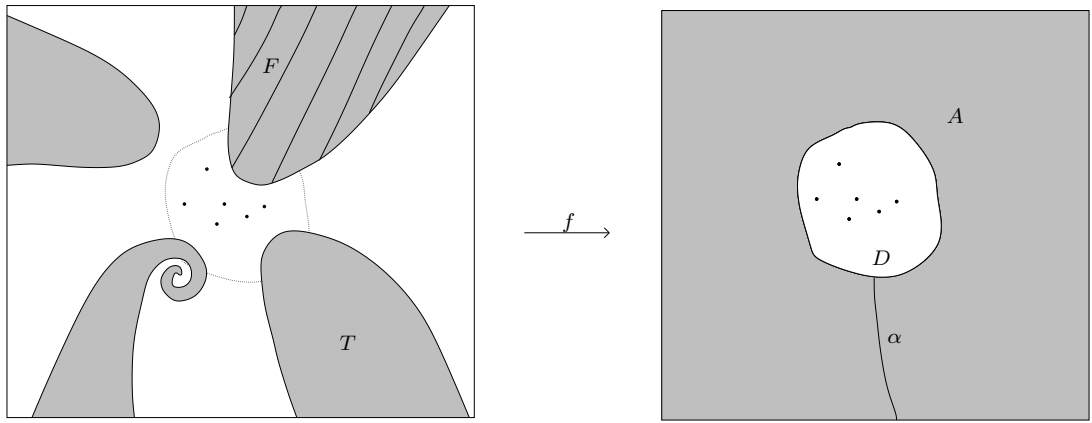}
\caption{Tracts and fundamental domains of a function $f\in\B$.}
\end{figure}

If $z\in\C$ with $f^n (z)\in A$ for all $n\geq 0$, then the 
\emph{external address} of $z$ is the sequence
$\underline{s}=F_0 F_1 F_2 ...$ of fundamental domains defined by 
$f^n (z)\in F_{n}$.  Note that the fact that $\alpha\cap T =\emptyset$ for 
any tract $T$ of $f$ guarantees that $f^n(z)$ does indeed belong to a 
(unique) fundamental domain. 
If $z\in I(f)$ then there exists an integer $n_0$ such that 
  $\vert f^n (z)\vert \in A$ for all $n\geq n_0$.
Let $\sigma$ denote the one-sided shift operator, i.e. 
$\sigma(F_0 F_1 F_2 ...) = F_1 F_2 F_3 ...\;$. 
We say that $\underline{s}$ is \emph{periodic} if 
$\sigma^n(\underline{s})=\underline{s}$ for some $n\geq 1$ and 
\emph{preperiodic} if 
 some image $\sigma^n(\underline{s})$ of $\underline{s}$, $n\geq 1$, 
is periodic but not $\underline{s}$ itself. 

\begin{Definition}[(Dynamic rays and ray tails)] \label{defn_ray}
A \emph{ray tail} of $f$ is an injective curve
\begin{eqnarray*}
g:[t_0,\infty)\rightarrow I(f)
\end{eqnarray*}
(where $t_0>0$) such that for each $n\in\N$, 
$\lim_{t\rightarrow\infty} f^{n}(g(t))=\infty$ and such that, as $n\to\infty$,
$f^n(g(t))\rightarrow\infty$  uniformly in $t$.

A \emph{dynamic ray} of $f$ is then a maximal injective curve 
 $g:(0,\infty)\rightarrow I(f)$ such that 
$g\vert_{[t_0,\infty)}$ is a ray tail for every $t_0>0$. 
\end{Definition}

If $g$ is a dynamic ray, then there exists $t_0>0$ such that, for each
 $n\geq 0$, the curve $f^n(g([t_0,\infty)))$ is contained in a fundamental
 domain $F_n$. The sequence $\underline{s}=F_0 F_1 F_2 \dots$ is called the
 \emph{external address of $g$}.

It is not difficult to check that any periodic ray, as defined in 
 Definition \ref{defn_periodic_rays}, is indeed a dynamic ray
 in the sense of Definition \ref{defn_ray} with a periodic external address.

Conversely, any dynamic ray $g$ with $f^n(g)\subset g$ is a 
 periodic ray in the sense of Definition \ref{defn_periodic_rays}, after 
 a suitable reparametrization. It is less obvious that any 
 dynamic ray with a periodic external address is also itself periodic.
However, if there were two dynamic rays, say $g_1$ and $g_2$, with the same 
external address, it would follow from the proof of \cite[Corollary 3.4]{rempe5}
that $g_1$ is a subset of $g_2$ or the other way around (the stated reference 
implies only that $g_1$ and $g_2$ intersect but the proof shows more, namely 
that one of the two rays would have to be a subset of the other one). 
It would then follow from \cite[Lemma 3.3]{rempe5}
that $g_1$ equals $g_2$.

\subsection{Geometrically finite maps}
\label{subs_gf}
Set
\begin{eqnarray*}
 P_{\J}:=P(f)\cap\J(f)\quad\textrm{and}\quad P_{\F}:=P(f)\cap\F(f).
\end{eqnarray*}

Recall that by Definition \ref{def_gf}, $f$ is called  geometrically finite
  if $S(f)\cap \F(f)$ is compact and $P_{\J}$ is finite.
  Note that every such map belongs to the class $\B$. Furthermore, 
since $P(f^n)=P(f)$ and $\F(f^n)=\F(f)$, it follows that every iterate
of a geometrically finite map is again geometrically finite.

Following McMullen \cite{mcmullen1}, a rational function $R$ is called 
 geometrically finite if $P(R)\cap\J(R)$ is finite. 
 Using classical results on dynamics of rational maps one can easily deduce 
 that the Fatou set of such a map is the union of finitely many 
 attracting and parabolic basins \cite[Chapter $6$]{mcmullen1}. 
 The following statement shows that the same holds for a 
 geometrically finite entire transcendental map. 

\begin{Proposition}
\label{poss_F_comp}
 Let $f$ be a geometrically finite entire transcendental function. 
  Then the Fatou set of $f$ is either empty or consists of 
  finitely many  attracting and parabolic basins. Furthermore, 
  every periodic cycle in the Julia set is repelling or parabolic. 
\end{Proposition}
 
\begin{proof}
First note that $f$ cannot have wandering domains. Indeed, if $W$ was
 a wandering domain, then all limits of orbits of points in $W$ would belong 
 to $P_{\J}$ \cite[Theorem]{bergweiler2}, and hence the iterates in $W$
 would converge locally uniformly to a single periodic orbit in $P_{\J}$.
 This orbit clearly cannot be repelling or parabolic. By a result of
 Perez-Marco \cite{perez-marco2}, this orbit also cannot be irrationally
 indifferent. Hence $f$ has no wandering domains. 
 Additionally, if $f$ had a Siegel disk, then its boundary 
 would be contained in $P(f)$ \cite[Theorem 7]{bergweiler1}. This is
 again impossible because $P_{\J}$ is finite, so
 $f$ has no Siegel disks. Thus the Fatou set is the union of attracting and
 parabolic basins.

The set of attracting and parabolic
 basins forms an open cover of the compact set $S(f)\cap\F(f)$. Hence there
 exist finitely many attracting and parabolic basins that cover 
 this set. On the other hand, every attracting or parabolic basin must
 contain at least one point of $S(f)$ \cite[Theorem 7]{bergweiler1}. This
 proves the first claim. 

Furthermore, if $z_0$ was a Cremer point of $f$, then there 
 would be a sequence $w_k$ of points in $P(f)$ converging nontrivially
 to $z_0$ \cite[Corollary 14.4]{milnor}. Since $P_{\J}$ is finite and
 $P_{\F}\cap J(f)$ consists of finitely many parabolic cycles,
 this is also impossible. 
\end{proof}
\begin{Remark*}
 In particular, if $f$ is geometrically finite, then $P(f)$ is bounded.
\end{Remark*}

 Recall that $f$ is called \emph{subhyperbolic} if it is geometrically finite
  and has no parabolic cycles. Note that an entire function is subhyperbolic
  if and only if $P_{\F}$ is compact and $P_{\J}$ is finite. 
  An entire function is called \emph{postsingularly finite} if $P(f)$ is
  finite. 
  The Fatou set of such a map is either empty or the union of finitely many
  superattracting basins.
 Clearly, every postsingularly finite map is subhyperbolic, and 
  in particular geometrically finite.

\section{Geometric constructions and dynamics}
\label{sec_3}
To prove our main theorem we will use a hyperbolic domain $U$ such that our function
$f$ is expanding with respect to the hyperbolic metric of $U$, 
and such that $U$ has simple topology.
Our requirements are formalized in the following definition.

\begin{Definition}[(Admissible expansion domain)]
\label{dfn_adm_exp_dom}
Let $f$ be an entire transcendental function and let $z_0$ be a fixed point of $f$
which is either repelling or parabolic. A domain $U=U(f,z_0)\subset\C$ is called an 
\emph{admissible expansion domain} of $f$ at $z_0$, if the following properties hold:
\begin{itemize}
 \item[$(a)$] $U\subset\C\setminus(P(f)\cup\{ z_0\})$ and $\infty$ is an isolated boundary 
point of $U$.
\item[$(b)$] $f^{-1}(U)\subsetneq U$.
\item[$(c)$] $U$ is finitely-connected. Furthermore, $U\neq \C\setminus P(f)$.
\item[$(d)$] The point $z_0\in\partial U$ is accessible from $U$. 
If $K_0$ is the component of $\C\setminus U$ containing $z_0$, then
$K_0\setminus \{ z_0\}$ has finitely many components. 
\end{itemize}
\end{Definition}

Note that every map for which there is an admissible expansion domain at some 
repelling or parabolic fixed point must have a bounded postsingular set.
Furthermore, it follows from Definition \ref{dfn_adm_exp_dom} and 
\cite[Theorem 7]{bergweiler1} that, if $z_0\in\C$ is a 
puncture of $U$, then $z_0$ is a repelling fixed point of $f$.
We will show later that, if $P(f)$ is finite, then we can choose 
$\C\setminus U$ to be finite as well.

Also note that if $U$ is an admissible expansion domain of $f$ at $z_0$, then 
$U$ is also an admissible expansion domain of $f^n$ at $z_0$.

Let $U$ be an admissible expansion domain of $f$. 
Observe that by Lemma \ref{lem_U_hyp}, $U$ is hyperbolic. 
By definition, there exists at least one point 
$w\in \C\setminus (U\cup P(f))$.
Note that $w$ has has infinitely many
preimages under $f$; otherwise, 
$w$ would be a (Picard) exceptional value but every such point 
is also an asymptotic value of $f$
\cite[Chapter 5, Theorem 1.1]{gold_ost}.

\begin{sa}
Throughout Section \ref{subs_L} and \ref{subsec_33} we will assume 
that $f$ is an entire transcendental map and 
$z_0$ is a repelling or parabolic fixed point of $f$ with 
an admissible expansion domain, denoted by $U$.
\end{sa} 

As mentioned in the introduction, we will prove that Theorem \ref{thm1}
holds when we replace the condition that $f$ is geometrically finite by 
the property that every iterate of $f$ has an admissible expansion domain at every 
repelling or parabolic fixed point. 
It will become clear that the existence of admissible expansion domains is what is 
essential for our idea to work. 
We will show in section \ref{subs_aed_gf}
that \emph{every} iterate of a geometrically finite map has an 
admissible expansion domain at \emph{any} of its repelling  
or parabolic fixed points;
we will also give an example of a map which is not geometrically finite 
but to which our methods still apply.

We note that we do not require that $f$ has finite order.
 Indeed,
 we will prove a general combinatorial statement (Theorem \ref{thm_finite}),
 requiring only
 our standing assumption. Theorem \ref{thm1} will then be deduced by
 applying this result to an iterate of the original geometrically finite function 
of finite order.

\subsection{Legs and the leg map $\L$}
\label{subs_L}
\begin{Definition}
\label{def_spider}
A \emph{leg} is an injective curve 
$\gamma:[0,\infty]\rightarrow U\cup\lbrace z_0, \infty\rbrace$ such that 
\begin{itemize}
\item[(i)]$\gamma\vert_{(0,\infty)}\subset U$,
\item[(ii)]$\gamma(0)=z_0$ and $\gamma(\infty)=\infty$.
\end{itemize}
\end{Definition} 
  
Two legs $\gamma_1$ and $\gamma_2$ are called 
\emph{equivalent} ($\gamma_1\sim\gamma_2$)
if they are homotopic in $U$ relative to the set of endpoints 
$\lbrace z_0,\infty\rbrace$. For a leg $\gamma$ we will 
denote its equivalence class by $[\gamma]$. 

By assumption, $z_0$ is not a critical point of $f$, 
so every leg ending at $z_0$ has a unique preimage curve 
ending at $z_0$ and this is again a leg. The map 
which assigns such a pullback to each leg $\gamma$ will be called the \emph{leg map} 
and denoted by $\L$. As usual,
we will denote the $n$-th iterate of $\L$ by $\L^n$.

It follows from the Homotopy Lifting Property that if 
$\gamma_1\sim\gamma_2$, then this also holds for their images, 
i.e.  $\L(\gamma_1)\sim\L(\gamma_2)$. Hence, 
\emph{the leg map $\L$ descends to a map on the set of equivalence classes of legs}.

We will often replace pieces of arbitrary legs by 
pieces of geodesics in their homotopy classes, 
which is possible by Proposition \ref{roy}. 
We will call a leg that is a geodesic with respect to the 
hyperbolic metric on $U$ a \emph{geodesic leg}.

We are now able to formulate the main result of Section \ref{sec_3}.
\begin{Theorem}
\label{thm_finite}
Let $\gamma$ be a leg. Then there exist integers $m$ and $n$ 
such that $\L^m(\gamma)\sim\L^n(\gamma)$.
\end{Theorem}
The proof of this theorem will be given at the end of Section \ref{subsec_33}.

\begin{figure}
\centering
\begin{minipage}{.3\linewidth}
\includegraphics[width=\linewidth]{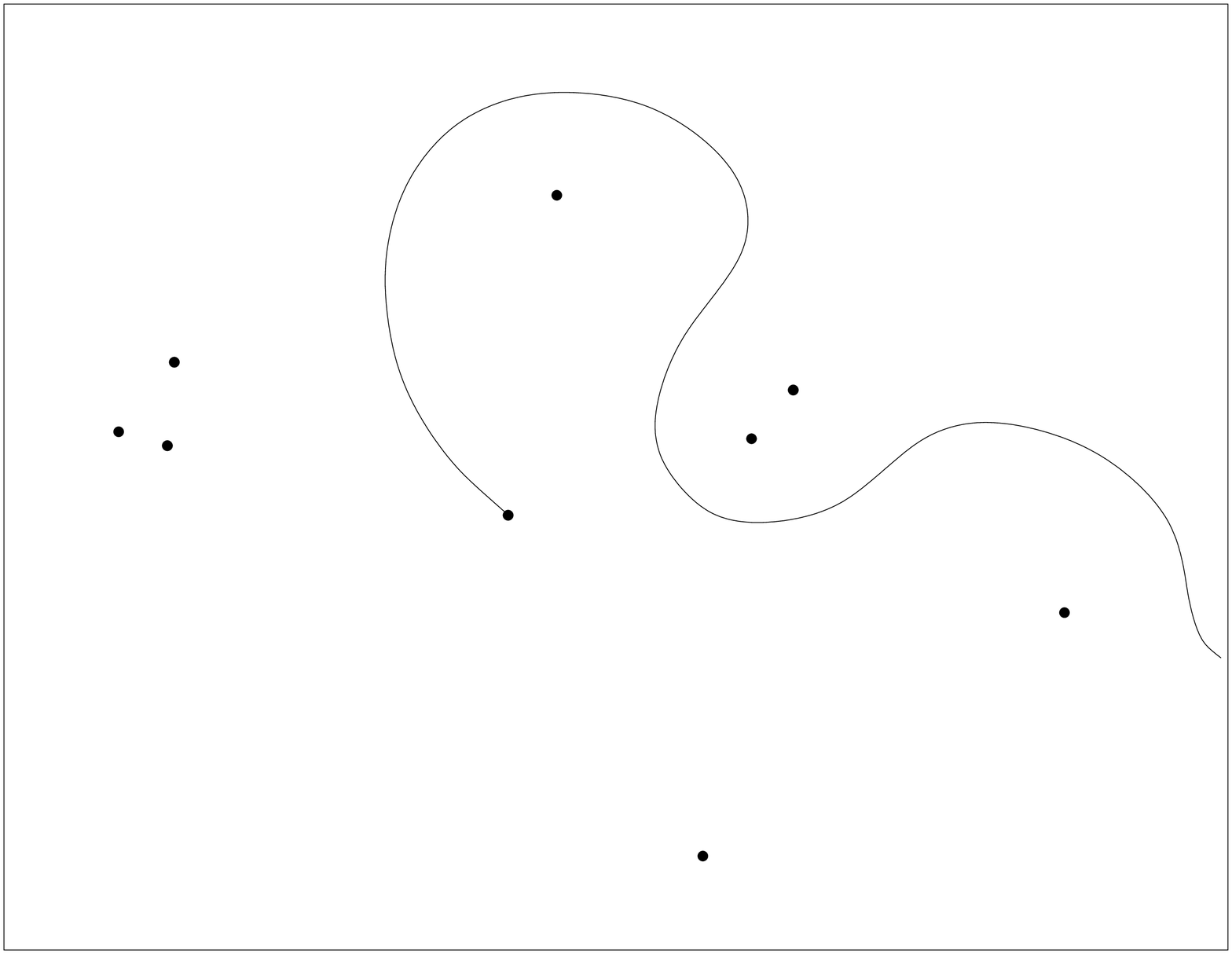}
\end{minipage}
\hspace{.025\linewidth}
\begin{minipage}{.3\linewidth}
\includegraphics[width=\linewidth]{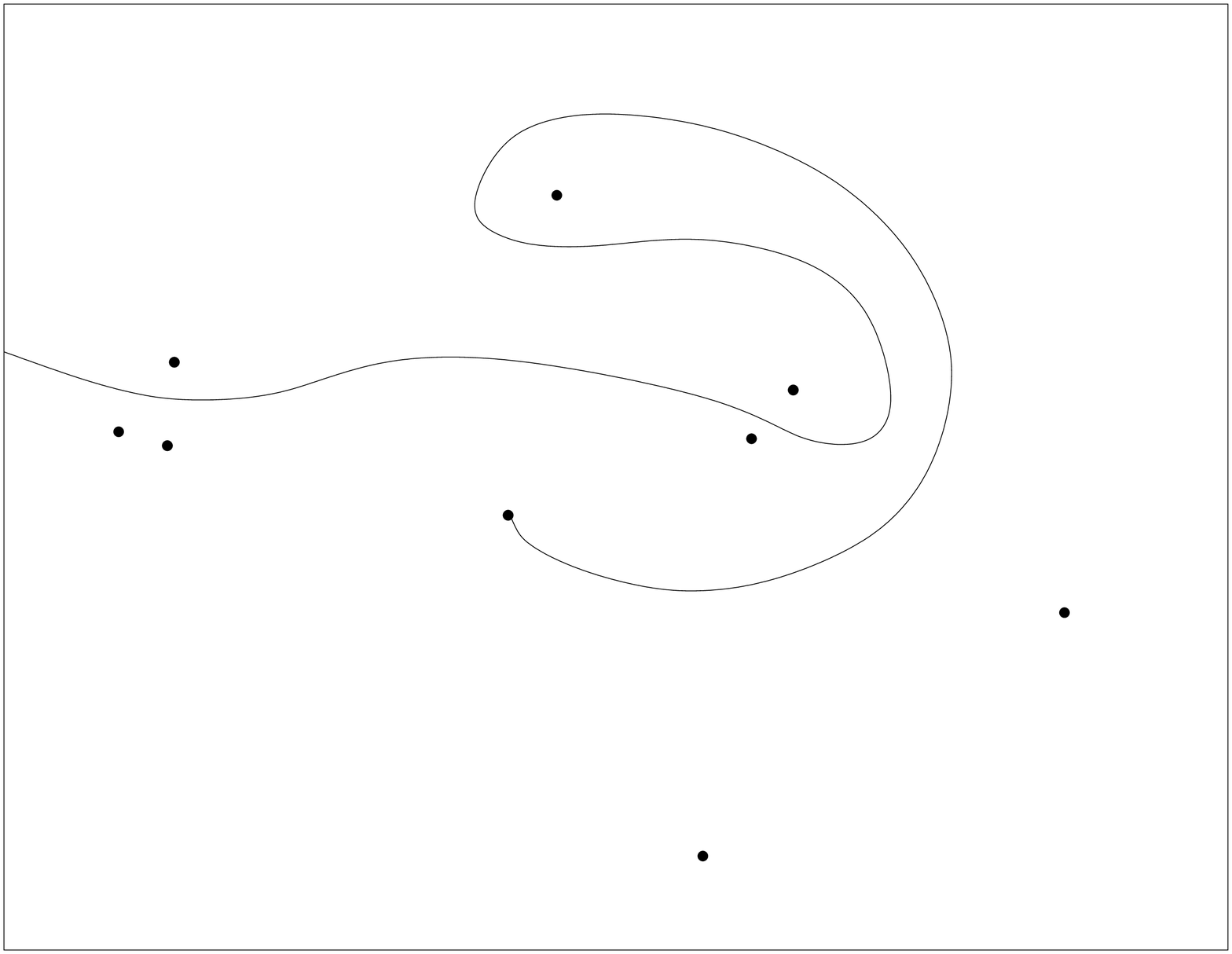}
\end{minipage}
\hspace{.025\linewidth}
\begin{minipage}{.3\linewidth}
 \includegraphics[width=\linewidth]{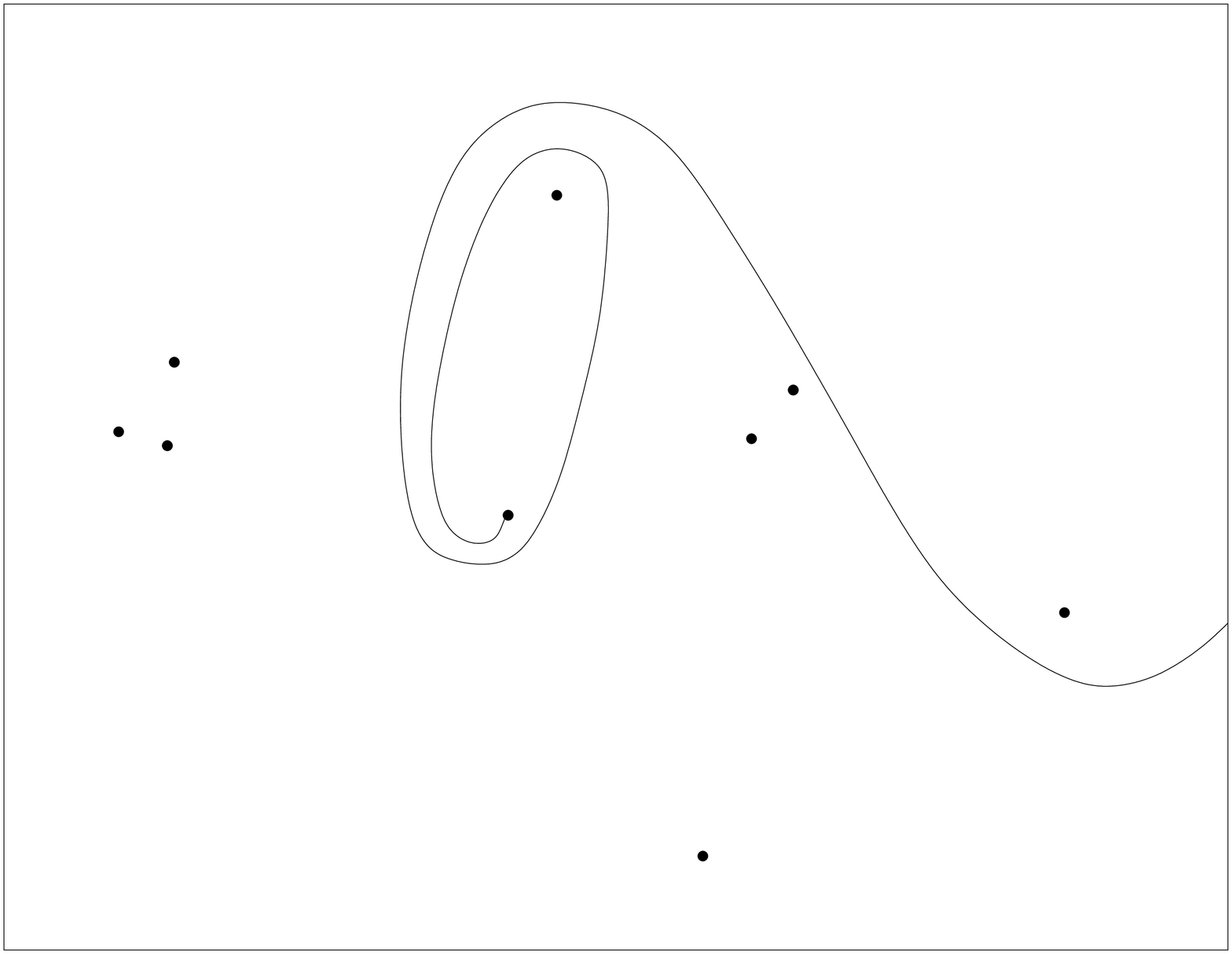}
\end{minipage}
\caption{Legs belonging to three different equivalence classes.}
\end{figure}

\subsection{Iteration of $\L$ and a finiteness statement}
\label{subsec_33}

From now on, let us denote the density of the hyperbolic metric
on the admissible expansion domain $U$ of $f$ at $z_0$ by $\rho_U(z)$. 
We define 
\begin{align*}
V:=f^{-1}(U). 
\end{align*}
Note that $V$ does not have to be connected.  
Every component $V_i$ of $V$ is again a hyperbolic domain with corresponding density map  
$\rho_{V_i}$. If $z\in V$, then $z$ lies in a unique component $V_i$ and for simplicity,
we will denote the density of the hyperbolic metric at $z$ by $\rho_V(z)$.

\begin{Proposition}
\label{prop_Y}
Let $Y\subset U$ be a compact connected set with finitely many boundary components. 
Then there exists a constant $\eta<1$ such that 
\begin{align*}
\rho_U(z)<\eta\cdot \rho_V(z)\quad\text{holds for all }z\in f^{-1}(Y).
\end{align*}
\end{Proposition}

\begin{proof}
For any compact subset of $f^{-1}(Y)$ there is certainly such a constant 
$\eta$, since $\partial Y$ has only finitely many components. 
Hence we have to consider only sufficiently 
large points $\tilde{z}\in f^{-1}(Y)$. 

Let $w\in\C\setminus (U\cup P(f))$ be a non-exceptional value of $f$. 
Then $w$ has infinitely many preimages under $f$ and all but finitely many of them 
are contained in $U\backslash V$.

\begin{claim*}
There exists a sequence 
$w_j\in U\backslash V$ and a constant $K>1$ such that 
$\vert w_{j+1}\vert\leq K\vert w_j\vert$ and $f(w_j)=w$ 
holds for all $j\in\N$.
\end{claim*}

\begin{proof of claim}A sketch of a proof can be found in 
\cite[proof of Lemma 5.1]{rempe3}. 
For completeness we will elaborate the arguments given in \cite{rempe3}.

Let $\gamma\subset U$ be a Jordan curve, such that the 
bounded component of $\C\backslash\gamma$ contains 
$S(f)$ but not $w$, and let $U_{\infty}$ denote the unbounded component of 
$\C\backslash\gamma$. Then $f^{-1}(U_{\infty})$ is a countable union 
of tracts $T_i$ and $f\vert_{T_i} : T_i\rightarrow U_{\infty}$ 
is a universal covering for every $i$. Let us pick a tract $T_0$. 
Since $w\in U_{\infty}$, there is an infinite sequence $w_i$ of preimages 
of $w$ in $T_0$, such that the distance $d_{T_0}(w_i,w_{i+1})$ 
measured in the hyperbolic metric of $T_0$ is constant. 

We can assume w.l.o.g. that $0\not\in T_0$, and so 
by equation \ref{hyp_est} we obtain  
\begin{eqnarray*}
\rho_{T_0}(z)\geq\frac{1}{2\vert z\vert}.
\end{eqnarray*}
Let $A:= d_{T_0}(w_i, w_{i+1})$. It follows that
\begin{eqnarray*}
A=\inf_{\gamma}\int_{t_0} ^{t_1} \rho_{T_0}(\gamma(t))\cdot\vert\gamma^{'}(t)\vert dt 
\geq\inf_{\gamma}\int_{t_0} ^{t_1}\frac{\vert\gamma^{'}(t)\vert }{2\vert\gamma(t)\vert}dt
 =\frac{1}{2}(\log\vert w_{i+1}\vert - \log\vert w_i\vert),
\end{eqnarray*}
where $\gamma:\left[ t_0, t_1\right] \rightarrow T_0$ is any 
rectifiable curve that connects $w_0$ and $w_1$. Hence 
\begin{eqnarray*}
\vert w_{i+1}\vert \leq \e^{2A}\vert w_i\vert.
\end{eqnarray*}
The claim now follows with $K=\e^{2A}>1$.
\end{proof of claim}

Recall that $U$ contains a punctured disk at $\infty$, hence 
\begin{align*}
 \rho_U(z)\leq\Oo\left(\frac{1}{\vert z\vert\cdot\log\vert z\vert}\right)\text{ as }z\to \infty.
\end{align*}
On the other hand, $V\subset\C\setminus\{ w_n\}$, and it follows from 
\cite[Proposition 2.1]{rempe3} that
\begin{align*}
\rho_V(z)\geq\Oo\left(\frac{1}{\vert z\vert}\right)\text{ as }z\to\infty.
\end{align*}
Hence $\rho_U(z)/\rho_V(z)\to 0$ as $z\to\infty$ and the statement follows.

\end{proof}

\begin{Proposition}
\label{lindisk}
Assume that $z_0$ is an isolated boundary point of $U$. Then 
for every horosphere $H_{\eps}(\infty)$ there exists a 
horosphere $H_{\delta}(z_0)$ such that 
$H_{\delta}(z_0)\subset U\backslash\overline{H_{\eps}(\infty)\cup f^{-1}(H_{\eps}(\infty))}$ 
and $f(H_{\delta}(z_0))\supset H_{\delta}(z_0)$. 

Furthermore, $\delta$ can be replaced by any $\tilde{\delta}<\delta$.
\end{Proposition}

\begin{proof}
First recall that $z_0$ is neccessarily a repelling fixed point of $f$. 

The first statement is obvious since $z_0\notin\overline{H_{\eps}(\infty)}$ holds 
for any horosphere $H_{\eps}(\infty)$ and since $z_0$ is a fixed point of $f$. 

There exist a covering map $p:\D^{*}\rightarrow U$ and a constant $0<\tau<1$ such that 
$p$ maps $\D_{\tau}(0)\setminus \{ 0\}$ one-to-one
to the horosphere $H_{\tau}(z_0):= p(\D_{\tau}(0)\setminus \{ 0\})$ at $z_0$. 
By the Riemann Removable Singularity Theorem, 
the embedding $p\vert_{\D_{\tau}(0)\setminus\{ 0\}}$ can be continued holomorphically 
to $0$.

For any $\delta<\tau$ let $h_{\delta}(z_0)=p(\mathbb{S}_{\delta})$, where 
$\mathbb{S}_{\delta}:=\partial\D_{\delta}(0)$, 
and denote by $i(\delta)$ and $o(\delta)$ its inner and outer 
radius, respectively. Clearly, $i(\delta), o(\delta) \to 0$ as $\delta\to 0$ as well as  
\begin{eqnarray*}
 \frac{o(\delta)}{i(\delta)}\rightarrow 1\quad\textrm{as}\; \dist(h_{\delta}(z_0),z_0)\rightarrow 0.
\end{eqnarray*}
 
By composing $f$ with a linear transformation, we can assume that $z_0=0$, 
so the power series of the function $f$ has the form
\begin{eqnarray*}
 f(z)=\mu(0)\cdot z + O(z^2)
\end{eqnarray*}

in a neighbourhood of $0$, where $\mu(0)$ is the multiplier of $0$. 
Let $\vert z\vert=i(\delta)$. Then
\begin{eqnarray*}
\left\vert\frac{f(z)}{z}\right\vert\geq\frac{\vert\mu(0)\vert\cdot i(\delta) - O(i(\delta)^2)}{i(\delta)}=\vert\mu(0)\vert - O(i(\delta)) > \frac{o(\delta)}{i(\delta)}
\end{eqnarray*}
for every sufficiently small $\delta$. 
Hence every point $z\in h_{\delta}(0)$ is mapped outside 
the circle at $0$ with radius $o(\delta)$ and the statement follows.
\end{proof}

Recall that our goal in Section \ref{sec_3} is to prove 
Theorem \ref{thm_finite} which states that the iteration of 
the leg map produces only finitely many equivalence classes of legs. 
Since equivalence classes of legs arise, roughly speaking, 
by winding around components of $\partial U$, we want to find a 
compact subset $Y$ of $U$ so that producing additional homotopy 
implies increase of length of leg-pieces contained in $Y$. By 
choosing $Y$ so that Proposition \ref{prop_Y} applies, we can later use 
uniform contraction arguments to control the lengths of the considered 
pieces of legs.

Note that if $\gamma$ is \emph{any} leg, it is fairly impossible to make 
useful statements about the location of its iterated images $\L^n(\gamma)$ related to
an \emph{arbitrary} compact set $Y$. Only by constructing $Y$ carefully using hyperbolic
geometry and working with geodesic legs rather than arbitrary legs, we obtain  
additional tools that enable us to control the lengths of geodesic leg-pieces.
 
\begin{Theorem}
\label{thm_Y}
There exists a compact path-connected set $Y\subset U$ 
with finitely many boundary components 
such that:
\begin{itemize}
\item[$(a)$] If $g$ is a geodesic leg, then $g\cap Y$ is non-empty and connected. 
Furthermore, if $K_1$ and $K_2$ are distinct components of $\widehat{\C}\setminus U$, 
then $K_1$ and $K_2$ are contained in two distinct 
components of $\widehat{\C}\setminus Y$.
\item[$(b)$] Let $C(z_0)$ and $C(\infty)$ denote the components of $U\setminus Y$
that contain $z_0$ and $\infty$, respectively, as boundary points, 
and for a leg $\gamma$, denote by $\tilde{\ell}_U(\gamma)$ the hyperbolic length 
in $U$ of the longest subpiece of $\gamma$ connecting 
the boundaries of $C(z_0)$ and $C(\infty)$ in $U$.
Then there exists a constant $0<P<\infty$, such that  
if $g$ is a geodesic leg and $\gamma\in [g]$ is another leg,
then $\tilde{\ell}_U(g)\leq \tilde{\ell}_U(\gamma) + P$. 
\item[$(c)$] There exists a constant $0<M<\infty$, such that 
if $g$ is a geodesic leg, then there exists a leg $\gamma_1\in[\L(g)]$ with
$\tilde{\ell}_U(\gamma_1)\leq\ell_U(\L(g)\cap f^{-1}(Y))+M$.
\end{itemize}

\end{Theorem}

\begin{proof}
Let $p_0,\dots ,p_n$ be the punctures of $U$ including $\infty$, so let us assume that 
$p_n=\infty$. For every $i=0,\dots n$ 
choose a sufficiently small horosphere $H_{\delta_i}(p_i)$ 
which satisfies the conclusion of Proposition \ref{jorg}, and such that 
$\overline{H_{\delta_i}(p_i)}\cap \overline{H_{\delta_j}(p_j)}$ whenever $i\neq j$. 
Recall from Section \ref{subs_tracts} that
$f^{-1}(H_{\delta_n}(\infty))$ is a countable 
union of tracts $T_i$, and any compact subset of $U$ can intersect
only finitely many tracts. 
If $z_0$ is one of the punctures, say $z_0=p_0$, 
we also require that $H_{\delta_0}(z_0)$ 
satisfies the conclusion of Proposition \ref{lindisk}.
Define
\begin{align*}
Y_1:=U\setminus\bigcup_{i=0}^n H_{\delta_i}(p_i).
\end{align*}
\vspace{.2cm}
Case I: \emph{$\C\setminus U$ is finite}
\vspace{.2cm}
 
Let $Y:=Y_1$. Clearly,  
$Y$ is a compact path-connected set with finitely 
many boundary components. 

$(a)$: If $g$ is a geodesic leg, 
then $g$ does not intersect $h_{\delta_i}(p_i)$ for all $1\leq i\leq n-1$, 
while it intersects $h_{\delta_0}(z_0)$ and $h_{\delta_n}(\infty)$ exactly once 
\cite[Proposition 3.3.9]{hubbard},
so in particular $g\cap Y$ is non-empty and connected. Since every puncture of $U$ belongs 
to a unique component of $\widehat{\C}\setminus Y$, statement $(a)$ follows.

$(b)$: Observe that among all curves in a given homotopy class 
which connect the two horocycles $h_{\delta_0}(z_0)$ and $h_{\delta_n}(\infty)$, 
the unique geodesic realizes the 
smallest distance, hence the claim follows with $P=0$. 

$(c)$: If $g$ is any geodesic leg, then, by $(a)$, $g$ intersects 
$\partial C(z_0)$ and $\partial C(\infty)$ exactly once, while it does not 
enter any other horosphere. 
Also recall that by Proposition \ref{lindisk}, 
the inverse branch of $f$ that maps $z_0$ to itself 
maps $C(z_0)$ into itself. Hence the only components 
of $Y\setminus f^{-1}(Y)$ that might have non-empty intersection with 
$\L(g)\cap Y$ are domains that arise as intersection 
of $U\setminus H_{\delta_n}(\infty)$ and a tract $T$   
(a component of $f^{-1}(H_{\delta_n}(\infty))$), that 
eventually contains $\L(g)$.
Observe that 
$\L(g)$ intersects $\partial T$ in exactly one point, say $w_0=\L(g)(t_0)$,
while it is possible that $\L(g)$ has more than one intersection 
point with $h_{\delta_n}(\infty)$ lying in $T$ 
(see Figure \ref{bild5}(a)). Let $w_1=\L(g)(t_1)$ be the last intersection 
point of $\L(g)$ and $h_{\delta_n}(\infty)$. 

If $t_0\geq t_1$, then the longest subpiece of $\L(g)$ connecting 
$\partial C(z_0)$ and $h_{\delta_n}(\infty)$ is itself a subpiece 
of $\L(g)\cap f^{-1}(Y)$ and the claim follows with $\gamma_1=\L(g)$ and $M=0$.  

Otherwise, let $\L^{'}(g)$ denote the subpiece of $\L(g)$
connecting $w_0$ and $w_1$. Clearly, $\L^{'}(g)\subset T$. 
Since $g\vert_{H_{\delta_n}(\infty)}$ is a geodesic in
$H_{\delta_n}(\infty)$, it follows 
that $\L(g)\vert_T$ is a geodesic in $T$, hence 
$\L^{'}(g)$ is a subpiece of a geodesic in $T$ connecting
$w_0$ to $\infty$.
Recall that $\partial T$ is an analytic curve, hence $\ell_U(\L^{'}(g))$
depends continuously on the point $w_0$. Furthermore, the set of those points
$w\in\partial T$ for which the geodesic in $T$ from $w$ to $\infty$ 
intersects $h_{\delta_n}(\infty)$ is a compact subset of $U$. 
Together with the fact that only finitely many tracts intersect the 
set $Y$, this implies that there is a finite number $M$ such that 
$\ell_U(\L^{'}(g))\leq M$. 
Hence $\tilde{\ell}_U(\L(g))\leq\ell_U(\L(g)\cap f^{-1}(Y)) +M$.

\begin{figure}
\centering
\begin{minipage}{.49\linewidth}
\raggedleft
\framebox(200,150){
\psfrag{Lg}{$\L(g)$}
\psfrag{w0}{$w_0$}
\psfrag{w1}{$w_1$}
\psfrag{znull}{$z_0$}
\psfrag{horo}{$h_{\delta_n}(\infty)$}
\psfrag{T}{$T$}
\includegraphics[width=.94\linewidth]{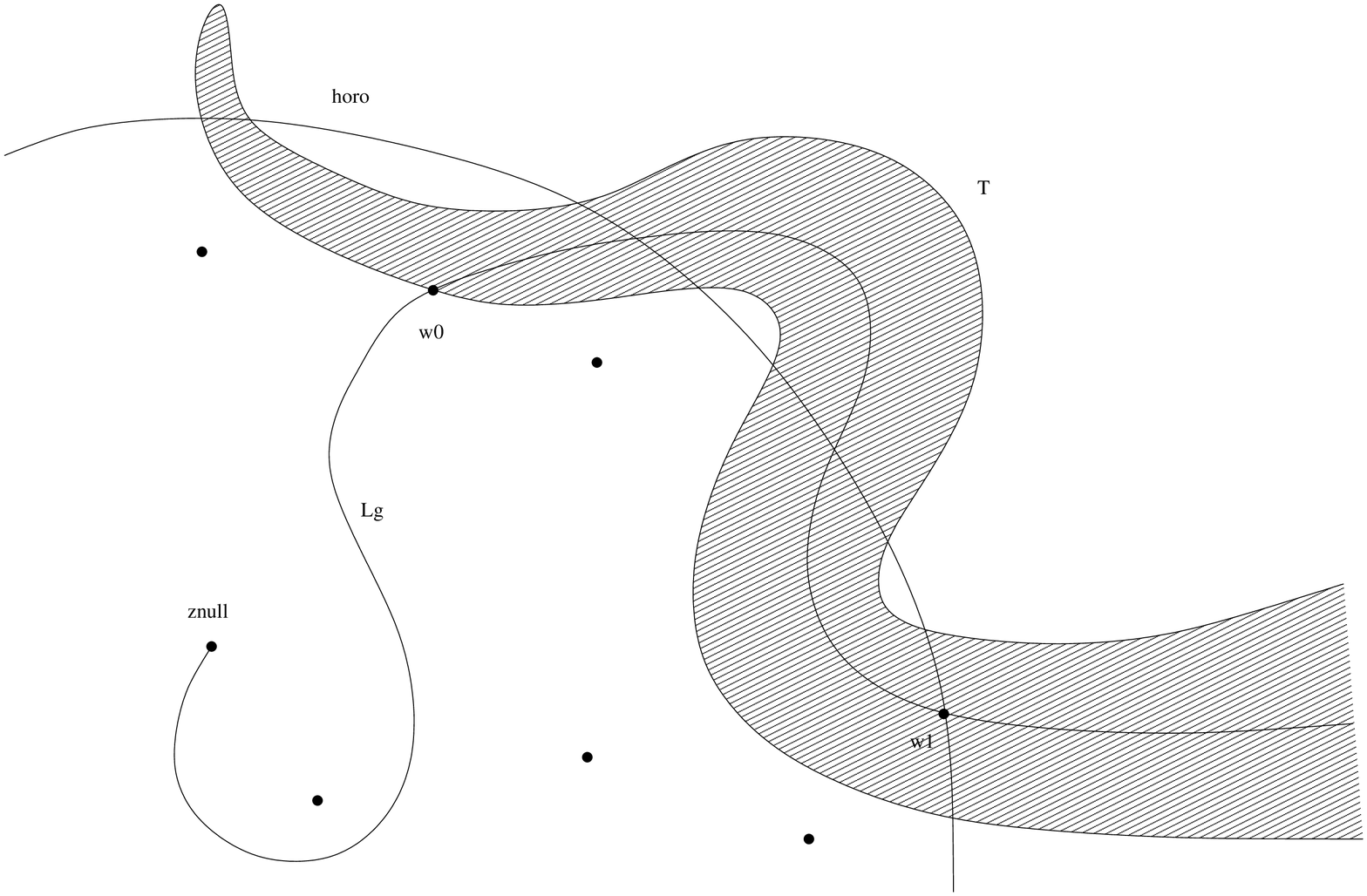}
}
\end{minipage}
\begin{minipage}{.49\linewidth}
\raggedright
\framebox(200,150){
\psfrag{wi}{$z_0$}
\psfrag{Ki1}{$K_0^1$}
\psfrag{Ki2}{$K_0^2$}
\psfrag{Ki3}{$K_0^3$}
\psfrag{Ji1}{$J_0^1$}
\psfrag{Ji2}{$J_0^2$}
\psfrag{Ji3}{$J_0^3$}
\includegraphics[width=.85\linewidth]{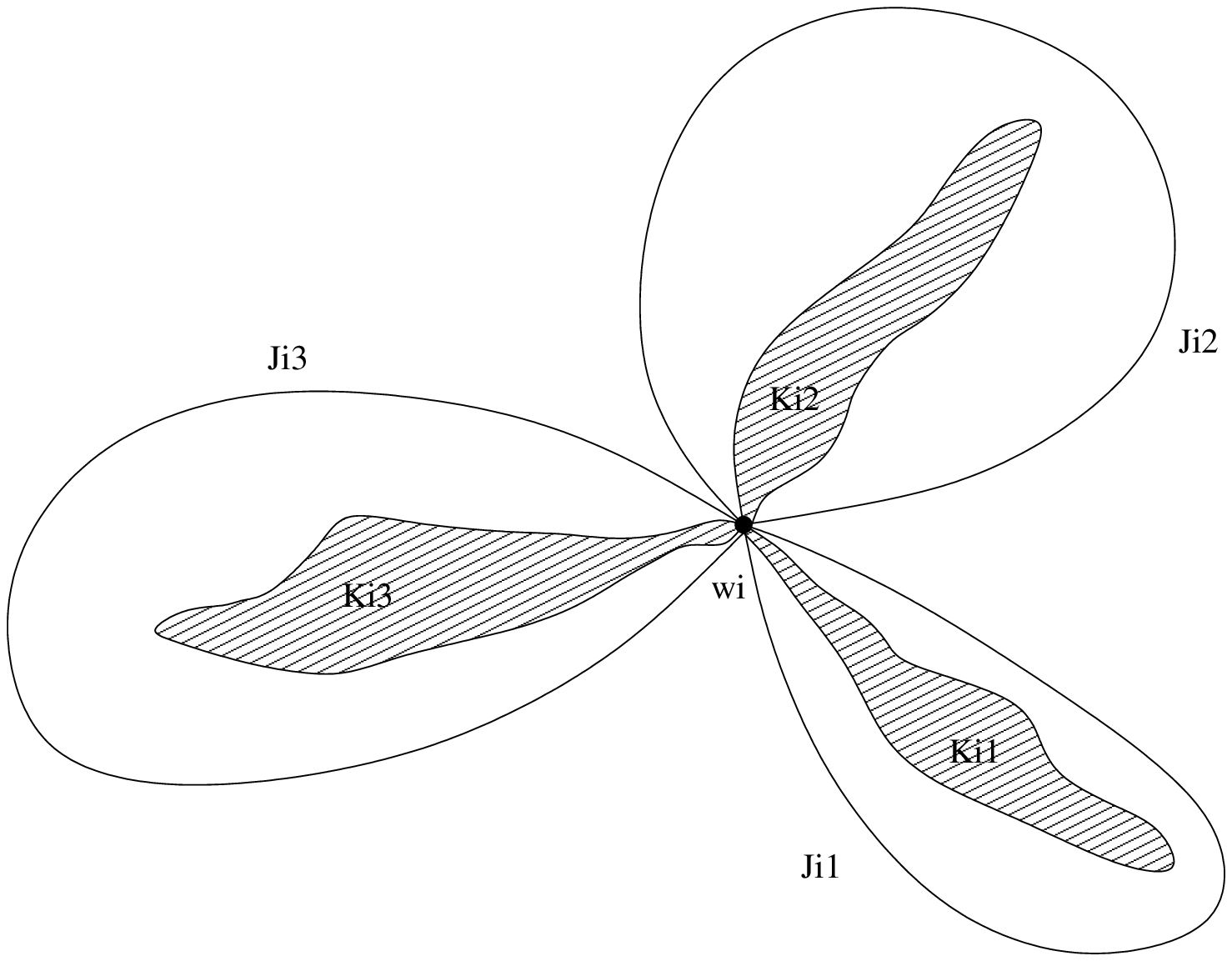}
}
\end{minipage}
\caption{(a): The image $\L(g)$ of a geodesic leg $g$: $\L(g)$ intersects $\partial T$ 
exactly once, while it can have more than one intersection point with $h_{\delta_n}(\infty)$,  
(b): The point $z_0$ is a non-isolated boundary point contained in the component 
$K_0$ of $\C\setminus U$: 
the separating curves $J_0^j$ intersect only in $z_0$.}
\label{bild5}
\end{figure}

\vspace{.2cm}
Case II: \emph{$\C\setminus U$ is infinite}
 \vspace{.2cm}

It follows from Definition \ref{dfn_adm_exp_dom}$(a),(c)$ 
that $\C\setminus U$ is a finite union 
of compact sets. Let $K_1,\dots K_m$ be those components of $\C\setminus U$ 
that are not punctures of $U$. 

Now let $K_i$ be a component such that $K_i\cap\{ z_0\} =\emptyset$.
By the Plane Separation Theorem \cite[Chapter VI, Theorem 3.1]{whyburn}
there exists a simple closed curve $J_i$, entirely contained in $U$, 
which separates $K_i$ from any other component of $\widehat{\C}\setminus U$. 
By \cite[Proposition 3.3.8, Proposition 3.3.9]{hubbard}, 
there is a unique geodesic $\alpha_i$  
which is a simple closed curve homotopic to $J_i$. 
Denote by $\hat{\alpha}_i$ the component of $U\backslash \alpha_i$
whose boundary consists of $\alpha_i\cup \partial K_i$. 

Let 
\begin{align*}
Y_2:=Y_1\setminus\bigcup_{i=0}^m \hat{\alpha}_i,
\end{align*}
where the union is taken over 
all components $K_i$ of $\C\setminus U$ such that 
$K_i\cap\{ z_0\}=\emptyset$. 
By \cite[Proposition 3.3.9]{hubbard} any 
two geodesics $\alpha_i\neq\alpha_j$ are disjoint 
and by our initial choice, 
any two horocycles $h_{\delta_j}(p_j)$ or geodesics $\alpha_i$ are 
disjoint as well. 
Hence the obtained set $Y_2$ is a subset of $U$ with finitely many 
boundary components, each of which is either a horocycle $h_{\delta_j}(p_j)$ or a simple closed 
geodesic $\alpha_i$. 

\vspace{.2cm}
Case IIa: \emph{$z_0$ is a puncture of $U$}
\vspace{.2cm}

Define $Y:=Y_2$. By construction,
$Y$ is a compact and path-connected set with finitely many boundary components. 

$(a)$: Let $g$ be a geodesic leg. Since the boundary of $Y$ consists 
of geodesics and horocycles and since $g$ intersects only the horocycles 
at $z_0$ and $\infty$, it follows that  $g\cap Y$ is connected. Furthermore, 
it follows from the previous construction that every component of $\partial Y$ 
surrounds exactly one component of $\partial U$, hence $(a)$ follows. 

$(b)$-$(c)$: These statements follow by exactly the same arguments 
as in case I.

\vspace{.2cm}
Case IIb: \emph{$z_0$ is not a puncture of $U$}
\vspace{.2cm}

Let $K_0$ be the component of $\C\setminus U$ that contains $z_0$. By 
Definition \ref{dfn_adm_exp_dom}$(d)$, $K_0\setminus\{ z_0\}$ has 
finitely many components, say $K_0^1,\dots ,K_0^l$. 
It follows from the Plane Separation Theorem that for every $j=1,\dots l$ 
there is a simple closed curve $J_0 ^j\subset U\cup\{ z_0\}$ 
that separates $K_0^j$ from 
every component of $\widehat{\C}\setminus U$ other than $K_0$, 
as well as from every $K_0^i$, where $i\neq j$. 
Furthermore, $J_0^j\cap\partial U =\{z_0\}$ 
(see Figure \ref{bild5}(b)). 
Each $J_0^j$ is homotopic relative the start- and endpoint $z_0$ 
to a unique geodesic $\beta^j$ in $U$, 
and any two such geodesics 
$\beta^j$ and $\beta^k$ intersect only in $z_0$.
For every $j=1,\dots ,l$ let $\hat{\beta}^j$ be the component of 
$U\setminus (\beta^j\cup\{ z_0\})$
bounded by $\partial K_0^j$, $\beta^j$ and $\{ z_0\}$ 
(see Figure \ref{bild6}(a)), and let 
\begin{align*}
Y_3=Y_2\backslash\bigcup_{j=1} ^l \hat{\beta}^j. 
\end{align*}
 
It follows that $Y_3\cap \partial U=\{ z_0\}$ and that $z_0$ is accessible through 
exactly $l$ sectors, each of which lies between two geodesics $\beta^j$ and $\beta^{j+1}$ 
(modulo $l$). 

Let $\lambda_j$, $j=1,\dots ,l$, be a collection of simple geodesic arcs, 
each of which connects a point in $\beta^j$ to a point in $\beta^{j+1}$, such that
 the domain $\Lambda_j$ bounded by $\lambda_j$, $\{ z_0\}$, $\beta^j$ and $\beta^{j+1}$ 
is a simply-connected subdomain of $U$ (see Figure \ref{bild6}(b)). Define 
\begin{align*}
 Y:=Y_3\backslash\bigcup_{j=1}^l \Lambda_i.
\end{align*}

$(a)$: The statement follows by exactly the 
same arguments as in case IIa. 

$(b)$: Let $C(z_0)$ denote the unique component of $\C\backslash Y$ 
that contains $z_0$ and set $P=\ell_U(\partial C(z_0))+\ell_U(h_{\delta_n}(\infty))$. 
Now, $g\cap Y$ does not necessarily realize the shortest distance 
between $\partial C(z_0)$ and $h_{\delta_n}(\infty)$ in its homotopy class;
still, if $\gamma$ is a leg in $[g]$, then there exists 
a component $\gamma^{'}$ of $\gamma\cap Y$ which is homotopic to $g\cap Y$
relative $\partial C(z_0)\cup h_{\delta_n}(\infty)$ and we obtain   
$\tilde{\ell}_U(g)=\ell_U(g\cap Y)\leq\ell_U(\gamma^{'})+ P\leq\tilde{\ell}_U(\gamma)+P$.

$(c)$: Recall that there is a unique $j\in\{ 1,\dots ,l\}$ such that 
$g$ itersects $\lambda_j$; their intersection point, say $s$, is unique 
and the piece of $g$ connecting $z_0$ and $s$ is entirely 
contained in $\Lambda_j$. 
Let $\tilde{\Lambda}_j$ be the component of $f^{-1}(\Lambda_j)$
such that $\L(g)$ intersects $\partial\tilde{\Lambda}_j$ and let $\tilde{s}$ be 
their unique intersection point. Observe that the piece of $\L(g)$ that connects 
$z_0$ and $\tilde{s}$ is entirely contained in $\tilde{\Lambda}_j$.

If $\L(g)\cap\tilde{\Lambda}_j\cap (U\setminus C(z_0))=\emptyset$, 
then the situation is reduced to the
previous case and we can choose $\gamma_1=\L(g)$.

Otherwise, we replace the subpiece of $\L(g)$ that connects $z_0$ and 
$\tilde{s}$ by the unique homotopic geodesic $\zeta$ of the hyperbolic domain 
$\tilde{\Lambda}_j$ connecting $z_0$ and $\tilde{s}$. 
If $\tilde{s}\in C(z_0)$, then we are only interested in  
the longest piece of $\zeta$ connecting two points in $\partial C(z_0)$, and 
otherwise in the longest piece of $\zeta$ 
that connects $\partial C(z_0)$ and $\tilde{s}$.
Again, by continuity and compactness arguments (as in the case of intersections 
with tracts), it follows that the length in $U$ 
of every such piece is globally bounded. 
The claim now follows from 
the fact that there are only 
finitely many domains $\tilde{\Lambda}_j$.
 
\end{proof}

\begin{figure}
\centering
\begin{minipage}{.49\linewidth}
\raggedleft
\framebox(200,150){
\psfrag{w0}{$w_0$}
\psfrag{K01}{$K_0^1$}
\psfrag{K02}{$K_0^2$}
\psfrag{K03}{$K_0^3$}
\psfrag{beta01}{$\beta^1$}
\psfrag{beta02}{$\beta^2$}
\psfrag{beta03}{$\beta^3$}
\psfrag{betahat01}{$\hat{\beta}^1$}
\includegraphics[width=.72\linewidth]{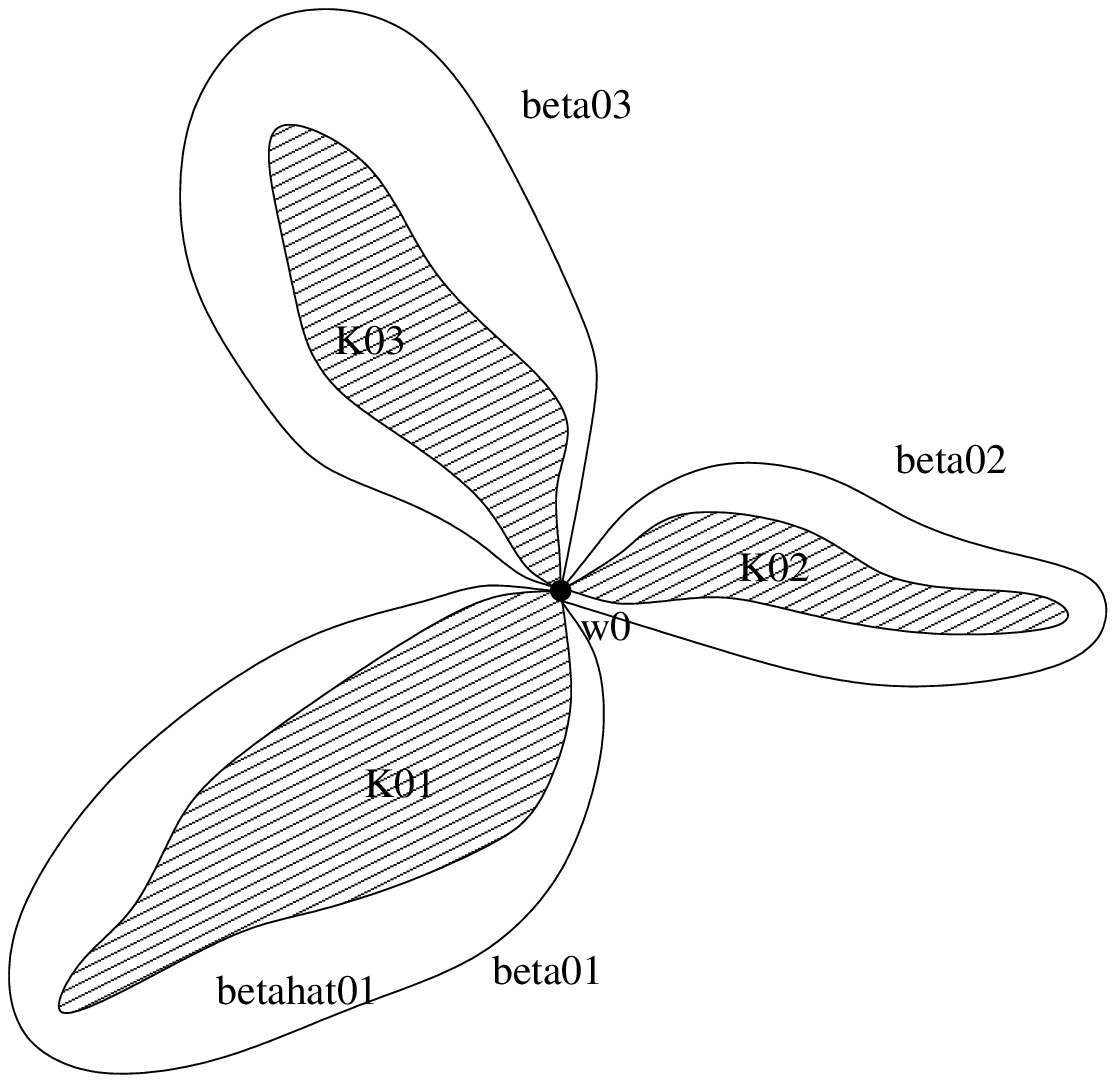}
}
\end{minipage}
\begin{minipage}{.49\linewidth}
\raggedright
\framebox(200,150){
\psfrag{lambda1}{$\lambda_1$}
\psfrag{lambda2}{$\lambda_2$}
\psfrag{beta1}{$\beta^1$}
\psfrag{beta2}{$\beta^2$}
\includegraphics[width=.9\linewidth]{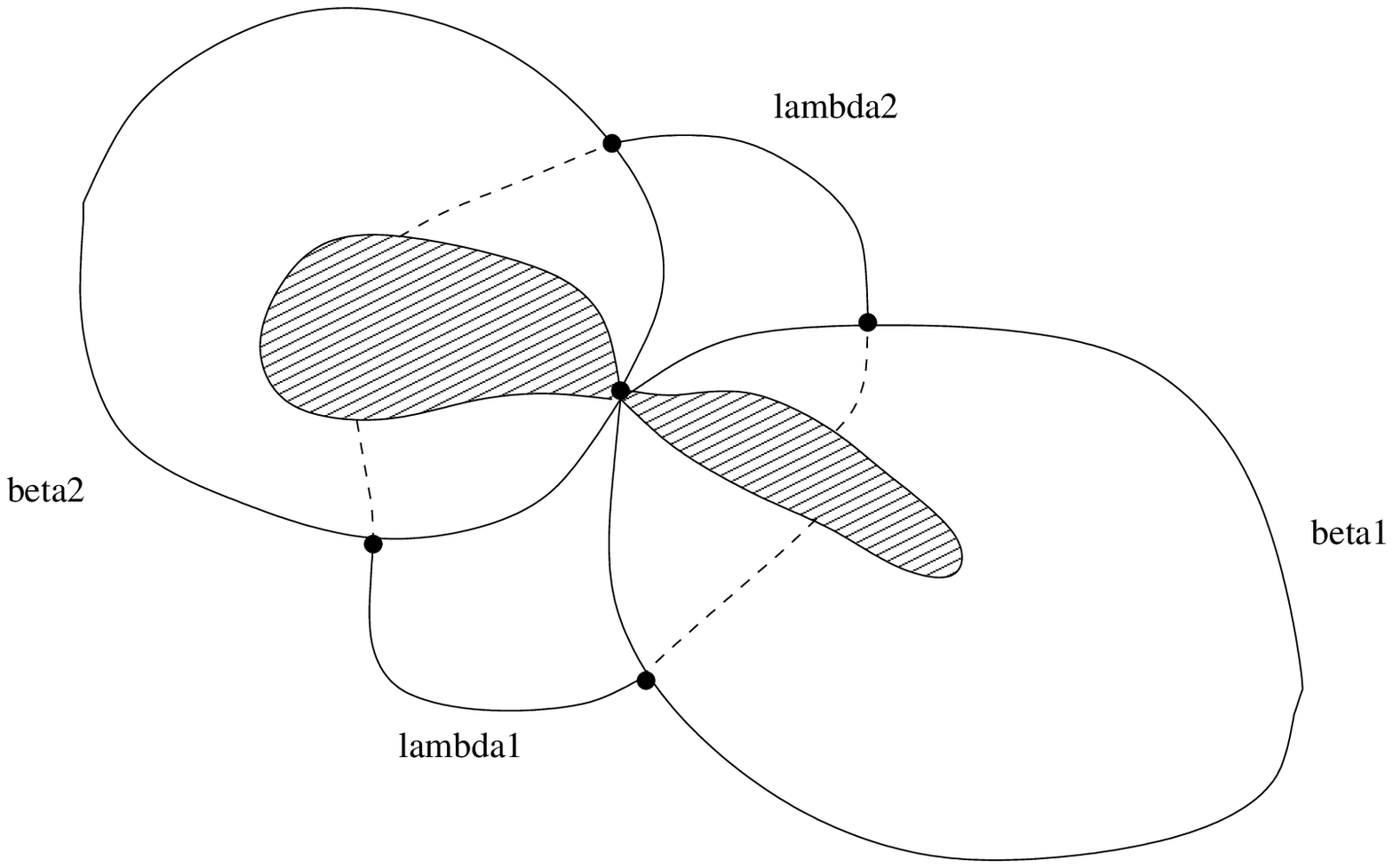}
}
\end{minipage}
\caption{(a): Every region $\hat{\beta}^j$ is bounded by $\beta^j\cup\partial K_0^j\cup\{ z_0\}$,  
(b): geodesics $\lambda_1$ and $\lambda_2$, one for each sector through 
which $z_0$ is accessible.}
\label{bild6}
\end{figure}

\begin{proof}[of Theorem \ref{thm_finite}]
Let $Y$ be a compact set that satisfies the conclusions of Theorem \ref{thm_Y}.
By Proposition \ref{prop_Y} there exists a constant $\eta<1$
such that $\rho_U(z)<\eta\cdot \rho_V(z)$ holds for all $z\in f^{-1}(Y)$. 
Hence if $c\subset f^{-1}(Y)$ is any rectifiable curve,
then $\ell_U(c)<\eta\cdot\ell_U(f(c))$. 

Let $g$ be a geodesic leg. By Theorem \ref{thm_Y}$(c)$ there exists 
a universal constant $M>0$ and
a leg $\gamma_1\in [\L(g)]$ such that 
$\tilde{\ell}_U(\gamma_1)\leq\ell_U(\L(g)\cap f^{-1}(Y))+M$. 
Together with the uniform contraction this yields the estimate
\begin{align*}
\tilde{\ell}_U(\gamma_1)\leq \ell_U(\L(g)\cap f^{-1}(Y)) + M <\eta\cdot\ell_U(g\cap Y) +M.
\end{align*}
Let $g_1\in [\L(g)]$ be a geodesic leg. By Theorem \ref{thm_Y}$(a)$, $g_1\cap Y$ 
is connected, hence $\tilde{\ell}_U(g_1)=\ell_U(g_1\cap Y)$. It then follows 
from Theorem \ref{thm_Y}$(b)$ 
that there exists a universal constant $P>0$ such that
$\tilde{\ell}_U(g_1)\leq\tilde{\ell}_U(\gamma_1)+P$. Altogether, we obtain
\begin{align*}
\ell_U(g_1\cap Y) = \tilde{\ell}_U(g_1)
\leq \tilde{\ell}_U(\gamma_1) +P \leq\eta\cdot\ell_U(g\cap Y) +M +P.
\end{align*}
By proceeding inductively it follows
that if $g_n\in [\L^n(g)]$ is the geodesic leg, then 
\begin{align*}
\ell_U(g_n \cap Y)<\eta ^n\cdot \ell_U(g\cap Y) + (P+M)\cdot\sum_{i=0} ^{n-1} \eta^i. 
\end{align*}
So in particular, if $L>\frac{P+M}{1-\eta}$ and $\ell_U(g\cap Y)<L$,  
then $\ell_U(g_n\cap Y)<L$.
Recall that by Theorem \ref{thm_Y} 
every component of $\widehat{\C}\setminus Y$ contains exactly 
one component of $\widehat{\C}\setminus U$, hence 
there can be only finitely many geodesic legs with globally bounded 
length.  
So for all $n\in\N$, the geodesic legs $g_n$ belong to only finitely many 
equivalence classes.
\end{proof}

\subsection{Admissible expansion domains of geometrically finite maps}
\label{subs_aed_gf}

We will now show that geometrically
finite maps have admissible expansion domains. 
Such maps provide us with many examples to which our main result applies. 
Still, there are functions that admit expansion domains but are not geometrically finite;
an example will be given at the end of this section.

\begin{Proposition}
\label{prop_U}
Let $f$ be a geometrically finite map and let $z_0$ be an arbitrary but fixed 
repelling or parabolic fixed point of $f$. Then $f$ has 
an admissible expansion domain $U$ at $z_0$.
Furthermore, if $f$ is postsingularly finite, then $U$ can be chosen such that 
$\C\setminus U$ is finite.
\end{Proposition}

\begin{proof}
Case I: \emph{$P(f)$ is finite}
\vspace{.2cm}

Recall that $f$ cannot have any parabolic cycles. 
If $z_0\notin P(f)$, then define $U:=\C\setminus (P(f)\cup\{z_0\})$. 
Otherwise, since every function in class $\B$ has 
infinitely many repelling fixed points \cite{langley}, 
there is a repelling fixed point $w$ of $f$ that belongs to $\C\setminus P(f)$.
In this case define $U:=\C\setminus (P(f)\cup\{ z_0, w\})$.
In both cases there is a point in $\C\setminus (U\cup P(f))$ that has 
preimages arbitrarily close to $\infty$.

The set $U$ is open, connected and finitely-connected,
$\infty$ is an isolated boundary point and $\C\setminus U$ 
is a finite union of points. Since $f(P(f))\subset P(f)$, 
it follows that $f^{-1}(U)\subset U$. 
But the set $f^{-1}(\C\setminus U)$ is not compact, hence it follows that  
$f^{-1}(P(f))\neq P(f)$ and $f^{-1}(U)\subsetneq U$. 
Furthermore, $z_0$ is an isolated boundary point
and hence accessible from $U$.

\vspace{.2cm}
Case II: \emph{$P(f)$ is infinite}
\vspace{.2cm}

Recall that in this case $\F(f)\neq\emptyset$.
\begin{claim*} There exists a full set $K$ whose boundary has only 
finitely many components such that $P_{\F}\subsetneq K\subset (\F(f)\cup\Par(f))$ 
and $f(K)\subsetneq K$.
\end{claim*}
\begin{proof of claim}
Denote by $S_a$ the set of those singular values of 
$f$ that are attracted by a cycle in $\Attr(f)$. 
For every point in $\Attr(f)$ we choose a linearising neighbourhood; 
let $A$ be their union. 
So $A$ is a finite union of Jordan domains and it satisfies
$\Attr(f)\subset A$ and $f(\overline{A})\subset A$. Since $S_a$ 
is a compact subset of the Fatou set, we have that $d_a:=\dist(S_a,\J(f))>0$. 
For each point $z\in S_a$ we pick a Euclidean disk $D(z)$ 
of radius less than $d_a/2$ centered at $z$. 
The union of these disks forms an open cover of $S_a$ 
and since the set is compact, there is a subcover 
consisting of finitely many disks, say $D(z_1),\dots ,D(z_n)$.
Let $C_a=\cup_{i=1}^n D(z_i)$. 
Each of the disks $D(z_i)$ is mapped into the set $A$ after 
finitely many iterations, hence there is an integer $n_a$ 
such that $f^{n_a}(C_a)\subset A$. Define $K_a$ to be the 
union of the sets $\overline{A}$ and 
$\overline{\bigcup_{j=0}^{n_a-1} f^j (C_a)}$. 

Denote by $S_p$ the set of those singular values of 
$f$ that are attracted by a cycle in $\Par(f)$ and let $q\in\Par(f)$. 
For simplicity let us assume that $q$ is a fixed point of multiplicity $n+1$,
where $n>0$; the periodic case is analogous. 
By the Parabolic Flower Theorem \cite[Theorem 10.7]{milnor} 
there are $n$ (attracting) petals of arbitrarily small diameter
attached to $q$. Let $A_1,\dots ,A_n$ be such a collection 
of petals at $q$. Then 
there is a positive number $\delta$ so that 
if $\vert f^n(z)-q\vert\leq\delta$ holds for all $n$, 
then $z$ is contained in some $A_i$ \cite[Chapter $1$, $\S 3.3$]{eremenko2}. 
Denote by $S_p ^q$ the set of all points in $S_p$ that converge 
to $q$ under iteration. 
For each point $z\in S_p ^q$ we choose a 
Euclidean disk $D(z)$ centered at $z$ with sufficiently small radius 
such that its closure is contained in the Fatou set.
As before, the union of these disks is a cover of $S_p ^q$ and 
it has a subcover consisting of finitely many disks, 
say $D(z_1),\dots ,D(z_m)$. 
Let $C_q:=\cup_{i=1}^m D(z_m)$. 
There is an integer $n_q$ such that $f^{n}(C_q)$ is contained in the 
$\delta$-neighbourhood of $q$ for all $n\geq n_q$. 
We now define the set $K^q$ to be the union of
$\bigcup_{i=1}^n \overline{A_i}$ and $\overline{\bigcup_{j=0}^{n_q-1} f^j(C_q)}$, 
and the set $K_p$ 
to be the union of all sets $K^q$, where $q\in\Par(f)$.

Observe that $\C\backslash (K_a\cup K_p)$ is not necessarily connected. 
Define $K$ to be the union of the set $K_a\cup K_p$ and the 
bounded components of its complement. 
Since $f(K_a\cup K_p)\subset (K_a\cup K_p)$, 
it follows that $f(\partial K)\subset K$. 
Let $K^0$ be a component of the interior of $K$. 
Then $K^0$ and $f(K^0)$ are bounded and so by the Open Mapping Theorem, 
$\partial f(K^0)\subset f(\partial K^0)\subset K$. Hence $f(K)\subset K$ 
and by Montel's Theorem, $\{ f^n\vert_{K^0}\}$ is a normal family. 
It follows that the interior of $K$ is contained in the Fatou set. 
Furthermore, it follows from our construction that $\partial K\cap\J(f)=\Par(f)$, 
$f(K)\subsetneq K$ and that $K$ has only finitely many boundary components,
yielding the statement of the claim.
\end{proof of claim}
We define 
\begin{eqnarray}
\label{def_U}
U:=\C\backslash (P(f)\cup \{ z_0\} \cup K).
\end{eqnarray}
Since $U$ is the complement of a full set, it follows 
that $U$ is a domain with $\infty$ as an 
isolated boundary point.
Also, $U\neq \C\setminus P(f)$ since $P_{\F}\subsetneq K$.
Furthermore, $f(P(f)\cup \{ z_0\} \cup K)\subsetneq (P(f)\cup \{ z_0\} \cup K)$, 
hence $f^{-1}(U)\subsetneq U$. Since $\partial K$ has 
finitely many components and since $P_{\J}$ is a finite set, 
it follows that $U$ is finitely-connected.  

If $z_0$ is repelling, then it is a puncture of $U$ and in particular
an accessible boundary point. 
If $z_0$ is parabolic, then it belongs to a non-trivial component $K_0$
of $K$, which is disjoint from the 
repulsion vectors at $z_0$, hence $z_0$ is accessible  \cite[Lemma $10.5$]{milnor}. 
By construction, $K_0$ is a union of finitely many petals at $z_0$, 
hence $K_0\setminus\{z_0\}$ has finitely many components.
\end{proof}

Since every iterate of a geometrically finite map is again geometrically finite,
we immediately obtain the following statement.

\begin{Corollary}
\label{cor_U}
 Let $f$ be a geometrically finite map and let $f^n$ be an iterate of $f$.
Then for every repelling or parabolic fixed point $z_n$ of $f^n$, there is an
admissible expansion domain of $f^n$ at $z_n$. 
\end{Corollary}

The map 
\begin{align*}
f(z) = \frac{12\pi^2}{5\pi^2 -48}\left( \frac{(\pi^2 -8)z + 2\pi^2}{z (4z-\pi^2)}\cos\sqrt{z} + \frac{2}{z}\right)
\end{align*}
 was introduced in \cite{bergweiler3} as an example of an entire transcendental function 
that has a completely invariant Fatou component $V$, which contains an 
indirect singularity in its boundary.
We will state here some of the properites of $f$; for more details see 
\cite{bergweiler3}.
 
The map $f$ has infinitely many critical values, all of which are 
contained in a closed interval $[0,y]\subset [0,\infty)$, and which
accumulate at the asymptotic value $0$. Furthermore,
$0$ is a parabolic fixed point of multiplier $1$ and $(0,\infty)$ is
contained in its basin of attraction $V$. So in 
particular, $S(f)\cap\F(f)$ is not a compact set and hence $f$ is not 
geometrically finite. 

Since $f$ maps  $[0,\infty)$ into itself and since every singular value
of $f$ converges to $0$, there is a compact intervall $[0,\tilde{y}]$
that contains $P(f)$ and is mapped by $f$ into itself.
It follows that if $z_0$ is any repelling or parabolic fixed point of 
$f$, then the domain 
$U=\C\setminus (\{ z_0\}\cup [0,\tilde{y}])$ is an 
admissible expansion domain of $f$ at $z_0$. 
Moreover, if $z_n$ is a repelling or parabolic fixed point 
of an iterate $f^n$, then 
$U=\C\setminus (\{ z_n\}\cup [0,\tilde{y}])$ is an 
admissible expansion domain of $f^n$ at $z_n$.

More generally, let $f_{\alpha}(z):=\alpha f(z)$.
There exists a real number $\alpha_0>1$ such that for all $1<\alpha<\alpha_0$
the map $f_{\alpha}$ has an attracting fixed point $x_{\alpha}>0$
whose basin of attraction $V_{\alpha}$ contains $(0,\infty)$.  
Since for every such $\alpha$ the map $f_{\alpha}$ has a repelling fixed point
at $0$, it follows that $0\in\partial V_{\alpha}$ and so again, 
$f_{\alpha}$ is not a geometrically finite map.
Without remarkable differences to the previous case, 
we can construct admissible expansion domains
for every iterate $f^n_{\alpha}$ at any of its repelling or parabolic 
fixed points. 

\section{Proof of the main theorem}
This section is devoted to the proof of our main result; together with 
Corollary \ref{cor_U} and the results from 
\cite{rrrs} it will imply Theorem \ref{thm1} stated at the beginning
of the article (see also Corollary \ref{cor_fin_comp}).

\begin{Theorem}
\label{maintheorem}
Let $f$ be an entire transcendental function, $z_0$ 
a repelling or parabolic fixed point of $f$ and assume that there is  
an admissible expansion domain $U$ of $f$ at $z_0$. If  
for any periodic external address $\underline{s}$ there exists a periodic  
ray of $f$ with address $\underline{s}$, 
then there is a periodic ray of $f$ landing at $z_0$.
\end{Theorem}

\begin{proof}
Let us pick a horosphere $H_{\delta}(\infty)$ in the admissible expansion domain $U$.
Recall that the preimage of $H_{\delta}(\infty)$ under $f$ 
is a countable union of tracts, which we will denote by $T_i$. 
Moreover, each tract $T_i$ can be split into fundamental domains, depending 
on the choice of a curve that connects $h_{\delta_n}(\infty)$ to $\infty$ without 
intersecting any of the tracts. Let us fix such a selection of fundamental 
domains $F_i$.  
It is necessary to give an idea of 
how to define an external address, which is respected by 
homotopies, for a leg $\gamma$. 

So let $\gamma$ be any leg. 
Note that $\gamma$ does not even need to intersect a tract. 
On the other hand $\L(\gamma)$ is eventually contained in a tract
and $\L^2(\gamma)$ is eventually contained in a fundamental domain. 
The application of this procedure to any other leg in $[\gamma]$ 
leads to the same fundamental domain.

Let $g$ be a geodesic leg.
By Theorem \ref{thm_finite} the equivalence class $[g]$ is
eventually periodic.
Since $U$ is also an admissible expansion domain of 
every $f^n$, we can assume, by passing to a suitable iterate, that 
$[g]$ is actually fixed.
We can also assume that $g$ is eventually contained 
in a fundamental domain, say $F_0$, and so are its images, 
since by the previous discussion 
this is true for all $\L^n(g)$ with $n\geq 2$.  
Hence we can assign to $g$ the fixed external address
\begin{eqnarray*}
 \underline{s}=\overline{F_0}=F_0 F_0 F_0\dots .
\end{eqnarray*}
By assumption, there exists a periodic ray 
$g_{\underline{s}}:(0,\infty)\rightarrow\C$
with address $\underline{s}$, hence there is a 
constant $\tau(g_{\underline{s}})>0$ 
such that $g_{\underline{s}}(t)\in F_0$ for all $t\geq\tau(g_{\underline{s}})$.
There is also a constant $\tau(g)>0$ so that $g(t)\in F_0$ for all 
$t\geq\tau(g)$. Let $\tau:=\tau(g)+\tau(g_{\underline{s}})$. We homotope
$g$ to a leg $\tilde{g}\in [g]$ by keeping $g\vert_{[0,\tau(g)]}$ fixed,
such that $\tilde{g}(\tau)=g_{\underline{s}}(\tau)$ and $\tilde{g}(t)\in F_0$ 
holds for all $t\geq\tau$. Note that this is always possible since 
every fundamental domain is a simply connected subset of $U$. 

Now, the tails $g\vert_{[\tau,\infty)}$ and 
$g_{\underline{s}}\vert_{[\tau,\infty)}$
are both entirely contained in the same fundamental domain $F_0$, 
hence we can replace $g\vert_{[\tau,\infty)}$ by the ray tail 
$g_{\underline{s}}\vert_{[\tau, \infty)}$, without changing the equivalence class.

When we apply $\L$ to the tails of $g$ and $g_{\underline{s}}$, then 
the resulting curves approach $\infty$ through the same fundamental domain $F_0$, 
and the same holds for the following iterates. 
Hence, after replacing a tail of $g$ by a tail of the dynamic ray and 
applying $\L$ we again obtain a leg eventually contained in $F_0$. 
Now, we want to show that in the limit, the iteration of $\L$ 
on such a leg yields a dynamic ray that lands at $z_0$.

Let $\tilde{g}(\sigma)=:x_0$ be a point on $\tilde{g}$ close to $z_0$. 
The sequence of iterated images of $x_0$ under the leg map 
(more precisely, under the corresponding inverse branch of $f$) 
will converge to the point $z_0$.
On the other hand, it follows from Pick's Theorem that 
the hyperbolic length of $\L^n(\tilde{g}\vert_{(\sigma,\tau)})$ decreases.
Hence the sequence $\L^n(\tilde{g}(\tau))$ also converges to $z_0$, 
which means that $g_{\underline{s}}$ lands at $z_0$.
\end{proof}

\begin{Remark*}
If $z$ is a point which is mapped to some 
fixed point $z_0$ of $f$ at which some periodic ray of $f$ lands, 
then $z$ itself is the landing point of a preperiodic ray of $f$.
\end{Remark*}

\begin{Corollary}
\label{cor_fin_comp}
Let $f=f_1\circ f_2\circ ...\circ f_n$ be a geometrically finite map, 
where $f_1,...,f_n$ belong to class $\B$ and have finite order of growth. 

Then, every repelling or parabolic periodic point of $f$ 
is the landing point of a periodic ray of $f$. 
In particular, every singular value in $\J(f)$ is the landing point of a dynamic ray.
\end{Corollary}

\begin{proof}
Let $z_0$ be an arbitrary but fixed repelling or parabolic periodic point of $f$
of period $n$. 
Then $z_0$ is a repelling or parabolic fixed point of $f^n$ and by 
Corollary \ref{cor_U} $f^n$ has an admissible expansion domain 
at $z_0$.
Note that $f^n$ also can be written 
as a finite composition of maps in class $\B$ with finite order. 

It follows from \cite[Theorem 2.4]{rempe1} 
and \cite[Proposition 4.5, Theorem 4.2]{rrrs}
that for every periodic external address there is a corresponding 
periodic ray of $f^n$. 
Theorem \ref{maintheorem} finally implies that $z_0$ is the landing 
point of a periodic ray of $f^n$. Since every periodic ray of $f^n$ is also 
a periodic ray of $f$  
and since $z_0$ was an arbitrary 
repelling or parabolic periodic point of $f$, we obtain the first claim.

The second claim follows now immediately, since every singular value in $\J(f)$ 
is eventually mapped onto a repelling or parabolic cycle.
\end{proof}

\begin{Remark*}
The existence of periodic rays for a map $f$ as in Corollary \ref{cor_fin_comp} 
also follows from a combination of results by Baranski \cite[Theorem C]{baranski} 
and Rempe \cite[Theorem $1.1$]{rempe3}. 
\end{Remark*}

\begin{Remark}
Let $f$ be a map as in Corollary \ref{cor_fin_comp}.
We can associate a \emph{dynamical partition}
to $f$, as done for exponential and cosine maps (see 
\cite[Section 4.1]{schleicher1}, \cite[p. 7]{schleicher2}), 
as follows:

Let us assume that all singular values of $f$ belong to $\J(f)$;
the case when $\F(f)\neq\emptyset$ uses similar ideas. 
Recall that $\J(f)=\C$.
By Corollary \ref{cor_fin_comp} every singular value $w_i$ of $f$ is the landing point 
of an eventually periodic ray, say $g_i$. 
Let 
\begin{align*}
\mathcal{D}:=\C\setminus\bigcup_i (g_i\cup w_i).
\end{align*}
Then $\mathcal{D}$ is a simply connected domain and 
$f: f^{-1}(\mathcal{D})\rightarrow\mathcal{D}$ is a covering map. 
We call $f^{-1}(\mathcal{D})$ a  \emph{dynamical partition} of $\C$. 
The components $I_i$ of $\mathcal{D}$ are simply-connected domains, 
called \emph{itinerary domains}, and the restriction 
$f\vert_{I_i}: I_i\to \mathcal{D}$ is a 
conformal map for any $i$. 

Such a partition is very useful for studying dynamics of $f$ in combinatorial terms; 
for instance, since no singular value of $f$ escapes to $\infty$, 
every dynamic ray of $f$ is contained in a unique itinerary domain.
\end{Remark}

\affiliationone{
H. Mihaljevi\'{c}-Brandt\\
Department of Mathematical Sciences,\\
The University of Liverpool,\\
Liverpool, L69 7ZL,\\
England, U.K. 
\email{helenam@liverpool.ac.uk\\}}

\end{document}